\documentclass[a4paper, 12 pt]{article}
\usepackage [a4paper,left=3.1 cm,bottom=2.7cm,right=3.1 cm,top=2.7cm]{geometry}
\usepackage[english]{babel}
\usepackage{amssymb}
\usepackage{amsmath,amsthm}
\usepackage{subcaption}
\usepackage{tabularx,array}
\usepackage{graphicx}
\usepackage{enumitem} 
\usepackage{mathtools}
\usepackage{xcolor}
\DeclareMathOperator{\sgn}{sgn}
\newtheorem{remark}{Remark}
\newtheorem{theorem}{Theorem}
\newtheorem{lemma}{Lemma}

\newtheorem{proposition}{Proposition}

\newcommand{\keywords}[1]{\par\addvspace\baselineskip
\noindent\enspace\ignorespaces#1}

\newcommand{\modch}{\color{black}}

\def\R{\mathbb{R}}

\begin{document}

\title{Optimal convergence rates for the invariant density estimation of jump-diffusion processes}

\author{Chiara Amorino\thanks{ Universit\'e du Luxembourg, L-4364 Esch-Sur-Alzette, Luxembourg. CA gratefully acknowledges financial support of ERC Consolidator Grant 815703 “STAMFORD: Statistical Methods for High Dimensional Diffusions”.} and Eulalia Nualart\thanks{Universitat Pompeu Fabra and Barcelona School of Economics, Department of Economics and Business, Ram\'on Trias Fargas 25-27, 08005
Barcelona, Spain. EN acknowledges support from the Spanish MINECO grant PGC2018-101643-B-I00 and
Ayudas Fundacion BBVA a Equipos de Investigaci\'on Cient\'ifica 2017.}} 

\maketitle

\begin{abstract}
 We aim at estimating the invariant density associated to a stochastic differential equation with jumps in low dimension, which is for $d=1$ and $d=2$.  We consider a class of {\modch fully non-linear} jump diffusion processes
whose invariant density belongs to some H\"older space.
Firstly,  in dimension one, we show that the kernel density estimator achieves the convergence rate $\frac{1}{T}$, which is the optimal rate in the absence of jumps. This improves the convergence rate obtained in \cite{Chapitre 4}, which depends on the Blumenthal-Getoor index for $d=1$ and is equal to $\frac{\log T}{T}$ for $d=2$.
Secondly, {\modch when the jump and diffusion coefficients are constant and the jumps are finite}, we show that is not possible to find an estimator with faster rates of estimation. Indeed, we get some  lower bounds with the same rates  $\{\frac{1}{T},\frac{\log T}{T}\}$ in the mono and bi-dimensional cases, respectively.
Finally, we obtain the asymptotic normality of the estimator in the one-dimensional case {\modch for the fully non-linear process}.
\end{abstract}

\keywords{{\bf Keywords:} Minimax risk, convergence rate, non-parametric statistics, ergodic diffusion with jumps, L\'evy driven SDE, invariant density estimation}

\section{Introduction}
Solutions to L\'evy-driven stochastic differential equations have recently attracted a lot of attention in the literature due to its many
applications in various areas such as finance, physics, and neuroscience. Indeed, it includes some important examples from finance such as the well-known Kou model in \cite{Kou02}, the Barndorff-Nielsen-Shephard model (\cite{BarShe01}), and the Merton model (\cite{Merton76}) to name just a few. An important example of application of jump-processes in neuroscience is the stochastic
Morris-Lecar neuron model presented in \cite{Mor_Lec}. As a consequence, statistical inference for jump processes has recently become an active domain of research.

We consider the process $X=(X_t)_{t \ge 0}$ solution to the following stochastic differential equation with jumps:
\begin{equation} \label{eq: model intro}
X_t= X_0 + \int_0^t b( X_s)ds + \int_0^t a(X_s)dB_s + \int_0^t \int_{\mathbb{R}^d_0}
\gamma(X_{s^-})z (\nu(ds,dz)-F(z) dz ds),
\end{equation}
where $(B_t)_{t \ge 0}$ is a $d$-dimensional Brownian motion and $\nu$ is a Poisson random measure on $\mathbb{R}_+ \times \mathbb{R}^d$ associated to a L\'evy process $(L_t)_{t \geq 0}$ with L\'evy density function $F$. We focus on the estimation of the invariant density $\mu$ associated to the jump-process solution to \eqref{eq: model intro} in low dimension, which is for $d=1$ and $d=2$. In particular, assuming that a continuous record of $(X_t)_{t \in [0,T]}$ is available, our goal is to propose a non-parametric kernel estimator for the estimation of the stationary measure and to discuss its convergence rate for large $T$.

The same framework has been considered in some recent papers such as \cite{Chapitre 4}, \cite{Strauch_new} (Section 5.2), and \cite{lowerbound}. In the first paper, it is shown that the kernel estimator achieves the following convergence rates for the pointwise estimation of the invariant density: $\frac{\log T}{T}$ for $d=2$ and $\frac{(\log T)^{(2 - \frac{(1 + \alpha)}{2}) \lor 1}}{T}$ for $d=1$ (where $\alpha$ is the Blumenthal-Getoor index).
We recall that, in the absence of jumps, the optimal convergence rate in the one-dimensional case is $\frac{1}{T}$, while the one found in \cite{Chapitre 4} depends on the jumps and  belongs to the interval $(\frac{\log T}{T}, \frac{(\log T)^\frac{3}{2}}{T})$. 

In this paper, we wonder if such a deterioration on the rate is because of the presence of jumps or the used approach. Indeed, our purpose is to look for a new approach to recover a better convergence rate in the one-dimensional case (hopefully the same as in the continuous case) and to discuss the optimality of such a rate. This new approach will also lead to the asymptotic normality of the proposed estimator. After that, we will discuss the optimality of the convergence rate in the bi-dimensional case. This will close the circle of the analysis of the convergence rates for the estimation of the invariant density of jump-diffusions, as the convergence  rates and their optimality in the case $d \ge 3$ have already been treated in detail in \cite{lowerbound}. 

Beyond these works, to our best knowledge, the literature 
concerning non-parametric estimation of diffusion processes with jumps is not wide. One of the few examples is given by Funke and Schmisser: in \cite{FunSch18} they investigate the non parametric adaptive estimation of the drift of an integrated jump diffusion process,  while in \cite{Sch19}, Schmisser deals with the non-parametric adaptive estimation of the coefficients of a jumps diffusion process.
To name other examples, in \cite{DL} the authors estimate in a non-parametric way the drift of a diffusion with jumps driven by a Hawkes process, while in \cite{Hawkes} the volatility and the jump coefficients are considered. 

On the other hand, the problem of invariant density estimation has been considered by many authors (see e.g. \cite{Ngu79}, \cite{Del80}, \cite{Bos98}, \cite{Zan01}, and \cite{Banon}) in several different frameworks: it is at the same time a long-standing problem and a highly active current topic of research. 
One of the reasons why the estimation of the invariant density has attracted the attention of many statisticians is the huge amount of numerical methods to which it is connected, the MCMC method above all. An approximation algorithm for the computation of the invariant density can be found for example in \cite{LamPag02} and \cite{Pan08}. Moreover, invariant distributions are essential for the analysis of the stability of stochastic differential systems (see e.g. \cite{Has80} and \cite{Banon}). 

In \cite{Banon}, \cite{BanNgu}, and \cite{Bosq9} some kernel estimators are used to estimate the marginal density of a continuous time process. When $\mu$ belongs to some H\"older class whose smoothness is $\beta$, they prove under some mixing conditions that their pointwise $L^2$ risk achieves the standard rate of convergence $T^{\frac{2 \beta}{2 \beta + 1}}$ and the rates are minimax in their framework. 
Castellana and Leadbetter proved in \cite{Cas_Lea} that, under condition \textbf{CL} below, the density can be estimated with the parametric rate $\frac{1}{T}$ by some non-parametric estimators (the kernel ones among them). 

{\modch In order to introduce condition \textbf{CL} it is necessary to request that the process $X$ belongs to a class of real processes with common marginal density $\mu$ with respect to the Lebesgue measure on $\R$ and such that the joint density of $(X_s, X_t)$ exists for all $s \neq t$, it is measurable and satisfies $\mu_{(X_s, X_t)} = \mu_{(X_t, X_s)} = \mu_{(X_0, X_{t-s})}$ and it is denoted by $\mu_{|t-s|}$ for all $s, t \in \R$. We also denote by $g_u$ the function $g_u(x,y) = \mu_u(x, y) - \mu(x) \mu(y)$. Then, condition \textbf{CL} writes as follows:}
\begin{itemize}
\item[\textbf{CL:}] $u \mapsto \left \| g_u \right \|_\infty$ is integrable on $(0, \infty)$ and $g_u(\cdot, \cdot)$ is continuous for each $u > 0$. 
\end{itemize}
In our context, $g_u(x, y)= \mu (x) p_u(x, y) - \mu(x) \mu(y)$, where $p_u(x, y)$ is the transition density. More precisely, they shed light to the fact that local irregularities of the sample paths provide some additional information. Indeed, if the joint distribution of $(X_0, X_t)$ is not too close to a singular distribution for $|t|$ small, then it is possible to achieve the superoptimal rate $\frac{1}{T}$ for the pointwise quadratic risk of the kernel estimator. Condition {\bf CL} can be verified for ergodic continuous diffusion processes (see \cite{Ver99} for sufficient conditions). 
The paper of Castellana and Leadbetter led to a lot of works regarding the
estimation of the common marginal distribution of a continuous time
process. In \cite{Bos97}, \cite{Bos98}, \cite{Che94}, \cite{Kut97}, and \cite{Bla96} several related results and examples can be found. 

An alternative to the kernel density estimator is given by the local time density estimator, which was proposed by Kutoyants in \cite{Kut98} in the
case of diffusion processes and was extended by Bosq and Davydov in \cite{BosDar99} to a more general context. The latest have proved that, under a condition which is mildly weaker than {\bf CL}, the mean squared error of the local time estimator reaches the full rate $\frac{1}{T}$. Leblanc built in \cite{Leb97} a wavelet estimator of a density belonging to some general Besov space and proved that, if the process is geometrically strong mixing and a condition like {\bf CL} is satisfied, then its $L^p$-integrated risk converges at rate $\frac{1}{T}$ as well. 
In \cite{Comte_Mer} the authors built a projection estimator and showed that
its $L^2$-integrated risk achieves the parametric rate $\frac{1}{T}$ under a condition named WCL, which is blandly different compared to {\bf CL}. 

\begin{itemize}
\item[\textbf{WCL:}] There exists a positive integrable function $k$ (defined on $\mathbb{R}$) such that
$$\sup_{y \in \mathbb{R}} \int_0^\infty \vert g_u(x, y) \vert du \le k(x), \qquad \text{ for all } x \in \mathbb{R}.$$
\end{itemize}

In this paper, we will show that our mono-dimensional jump-process satisfies a local irregularity condition {\bf WCL1} and an asymptotic independence condition {\bf WCL2} (see Proposition \ref{prop: WCL satisfied}), two conditions in which the original condition {\bf WCL} can be decomposed. In this way, it will be possible to show that the $L^2$ risk for the pointwise estimation of the invariant measure achieves the superoptimal rate $\frac{1}{T}$, using our kernel density estimator. Moreover, the same conditions will result in the asymptotic normality of the proposed estimator. Indeed, as we will see in the proof of Theorem \ref{th: as norm}, the main challenge in this part is to justify the use of dominated convergence theorem, which will ensured by conditions {\bf WCL1} and {\bf WCL2}.
We will find in particular that, for any collection $(x_i)_{1 \le i\le m}$ of real numbers, we have
$$
\sqrt{T}(\hat{\mu}_{h, T} (x_i) - \mu (x_i), \, 1 \le i \le m) \xrightarrow{\mathcal{D}} N^{(m)} (0, \Sigma^{(m)}) \text{ as } T \rightarrow \infty,
$$
where $\hat{\mu}_{h, T}$ is the kernel density estimator and
$$\Sigma^{(m)} := (\sigma(x_i, x_j))_{1 \le i, j \le m}, \qquad \sigma(x_i, x_j) := 2 \int_0^{\infty} g_u(x_i, x_j) du.$$
We remark that the precise form of the equation above allows us to
construct tests and confidence sets for the density. 

We have found the convergence rate $\left \{ \frac{1}{T}, \frac{\log T}{T} \right \}$ for the risk associated to our kernel density estimator for the estimation of the invariant density for $d=1$ and $d=2$.
Then, some questions naturally arise: are the convergence rates the best possible or is it possible to improve them by using other estimators? In order to answer, we consider a simpler model where both the volatility and the jump coefficient are constant and the intensity of the jumps is finite. Then, we look for a lower bound for the risk at a point $x \in \mathbb{R}^d$ defined as in equation \eqref{eq: def risk} below. The first idea is to use the two hypothesis method (see Section 2.3 in \cite{Ts}). To do that, the knowledge of the link between the drift $b$ and the invariant density $\mu_b$ is essential. {\modch In} absence of jumps such link is explicit, but in our context it is more challenging. As shown in \cite{Arnaud_Yoshida} and \cite{lowerbound}, it is possible to find the link knowing that the invariant measure has to satisfy $A^* \mu_b =0$, where $A^*$ is the adjoint of the generator of the considered diffusion. This method allows us to show that the superoptimal rate $\frac{1}{T}$ is the best possible for the estimation of the invariant density in $d=1$, but it fails in the bi-dimensional case (see Remark 1 below for details).
Finally, we use a finite number of hypotheses to prove a lower bound in the bi-dimensional case. This requires a detailed analysis of the Kullback divergence between the probability laws associated to the different hypotheses. Thanks to that, it is possible to recover the optimal rate $\frac{\log T}{T}$ in the two-dimensional case. 

The paper is organised as follows. In Section \ref{S: assumptions and results} we give the assumptions on our model and we provide our main results. Section \ref{S: prel} is devoted to state and prove some preliminary results needed for the proofs of the main results. To conclude, in Section \ref{S: proof main} we give the proof of Theorems \ref{th: upper bound}, \ref{th: as norm}, \ref{th: lower bound}, and \ref{th: lower d=2}, where our main results are gathered.

{\modch Throughout all the paper $c$ and $\lambda$ are constants that may change from line to line. Their dependence on $T$ or other fixed constants will be implied from the statements.}

\section{Model assumption and main results}{\label{S: assumptions and results}}
We consider the following stochastic differential equation with jumps
\begin{equation} \label{eq: model}
X_t= X_0 + \int_0^t b( X_s)ds + \int_0^t a(X_s)dB_s + \int_0^t \int_{\mathbb{R}^d_0}
\gamma(X_{s^-})z (\nu(ds,dz)-F(z) dz ds),
\end{equation}
where $t \geq 0$, $d \in \{1,2\}$, $\mathbb{R}^d_0=\mathbb{R}^d \backslash \left \{0 \right \}$, the initial condition $X_0$ is a $\R^d$-valued random variable, the coefficients 
$b:\R^d \rightarrow \R^d$, $a:\R^d \rightarrow \R^d\otimes \R^d$ and $\gamma:\R^d \rightarrow \R^d\otimes \R^d$ are measurable functions,
$(B_t)_{t \ge 0}$ is a $d$-dimensional  Brownian motion, and $\nu$ is a Poisson random measure on $\mathbb{R}_+ \times \mathbb{R}^d$ associated to a L\'evy process $(L_t)_{t \geq 0}$ with L\'evy density function $F$.
All sources of randomness are mutually  independent. 

We consider the following assumptions on  the coefficients and  on the  L\'evy density $F$:
\begin{enumerate} 
\item[{\bf A1}] The functions $b$, $\gamma$ and $aa^T$ are globally Lipschitz and bounded.
Moreover, $\inf_{x \in {\modch \mathbb{R}^d}}aa^T(x) \geq c$Id, for some constant $c>0$, where Id denotes the  $d \times d$ identity matrix and 
$\inf_{x \in { \modch \mathbb{R}^d}} \text{det} (\gamma(x))> 0$.

\item[{\bf A2}] 
 $\langle x, b(x) \rangle\le -c_1|x|+c_2$, for all $|x| \ge \rho$, for some $\rho,c_1,c_2 > 0$.

\item[{\bf A3}] Supp$(F)=\mathbb{R}^d_0$ and  for all 
$z \in \mathbb{R}^d_0$, $F(z) \le \frac{c_3}{|z|^{d + \alpha}}$, for some $\alpha \in (0,2), c_3>0$. 

\item[{\bf A4}] There exist $\epsilon_0 > 0$ and $c_4 > 0$ such that $\int_{\mathbb{R}^d_0}|z|^2 e^{\epsilon_0 |z|} F(z) dz \le c_4$. 

\item[{\bf A5}] If $\alpha =1$, $\int_{r <| z |< R} z F(z) dz =0$, for any $0 < r < R < \infty$.
\end{enumerate}

Assumption {\bf A1} ensures that equation (\ref{eq: model}) admits a unique c\`adl\`ag adapted solution $X=(X_t)_{t \geq 0}$ satisfying the strong Markov property, see e.g. \cite{Applebaum}. Moreover, it is shown in \cite[Lemma 2]{Chapitre 4} that if we further assume Assumptions {\bf A2-A4}, then the process $X$ is exponentially ergodic and exponentially $\beta$-mixing. Therefore {\modch the process is stationary and, in particular,} it has a unique invariant distribution $\pi$, which we assume it has a density $\mu$ with respect to the  Lebesgue measure. 
Finally, Assumption {\bf A5} ensures the existence of the transition density of $X$ denoted by $p_{t}(x,y)$ which satisfies the following upper bound (see \cite[Lemma 1]{Chapitre 4}): for all $T \geq 0$, there exist {\modch $c>0$ and $\lambda>0$} such that for any $t \in [0,T]$ and $x,y \in \mathbb{R}^d$,
\begin{equation} \label{pt}
p_t(x,y) \leq c \left(t^{-d/2} e^{-\lambda \frac{\vert y-x\vert^2}{t}}
+\frac{t}{(t^{1/2}+\vert y-x\vert)^{d+\alpha}} \right).
\end{equation} 

We assume that the process is observed continuously $X=(X_t)_{t \in  [0,T]}$ in a time interval $[0,T]$ such that $T$ tends to $\infty$. 
In the 
paper \cite{Chapitre 4} cited above, the nonparametric estimation of $\mu$ is studied via the  kernel estimator which is defined as follows.
We assume that  $\mu$ belongs to the  H\"older space $\mathcal{H}_d(\beta, \mathcal{L})$ where  $\beta=(\beta_1,\ldots,\beta_d)$, $\beta_i \ge 1$ and $\mathcal{L}=(\mathcal{L}_1,\ldots,\mathcal{L}_d)$, $\mathcal{L}_i>0$, which means that for all $i \in \{1,\ldots,d\}$, $k = 0,1,\ldots, \lfloor \beta_i \rfloor$ and $t \in \mathbb{R}$,
\begin{equation*} 
\left \| D^{(k)}_i \mu \right \|_\infty \le \mathcal{L}  \quad \text{ and } \quad \left \| D_i^{(\lfloor \beta_i \rfloor)}\mu(. + te_i) - D_i^{(\lfloor \beta_i \rfloor)}\mu(.) \right \|_\infty \le \mathcal{L}_i |t|^{\beta_i - \lfloor \beta_i \rfloor},
\end{equation*}
where $D^{(k)}_i$ denotes the $k$th order partial derivative of $\mu$ w.r.t the $i$th  component, 
$\lfloor \beta_i \rfloor$ is the integer part of $\beta_i$, and $e_1,\ldots,e_d$ is the canonical  basis of $\R^d$.
 {\modch That is, all the partial derivatives of $\mu$ up to order $\lfloor \beta \rfloor$ are bounded and the 
 $\lfloor \beta \rfloor$th partial derivative is H\"older continuous of order $\beta-\lfloor \beta \rfloor$ in any direction. We recall that it is natural in our context to assume that the invariant density belongs to a H\"older class as above. In fact, the proof of the bias bound (\ref{bv}) stated below gives a direct application of this assumption, see the proof of Proposition 2 in \cite{Chapitre 4}. Other examples of nonparametric estimation over H\"older classes can be found in \cite{Hop_Hof},  \cite{Jud_Nem}, \cite{Lep_Spo}, and \cite{Tsy}}. 

We set
\begin{equation*}
\hat{\mu}_{h,T}(x) = \frac{1}{T \prod_{i=1}^d h_i} \int_0^T \prod_{i=1}^d K\left(\frac{x_i - X^i_t}{h_i}\right) dt = : \frac{1}{T} \int_0^T \mathbb{K}_h(x - X_t) dt, 
\label{eq: def estimator}
\end{equation*}
where $x=(x_1,\ldots,x_d) \in \R^d$, $h=(h_1,\ldots,h_d)$ is a bandwidth and  $K: \mathbb{R} \rightarrow \mathbb{R}$ is a  kernel function satisfying 
$$\int_\mathbb{R} K(x) dx = 1, \quad \left \| K \right \|_\infty < \infty, \quad \mbox{supp}(K) \subset [-1, 1], \quad \int_\mathbb{R} K(x) x^i dx = 0,$$
for all $i \in \left \{ 0,\ldots, M \right \}$ with $M \geq \max_i \beta_i$. 

\vskip 12pt
We first consider equation (\ref{eq: model}) with $d=1$ and show that the kernel estimator reaches the optimal rate $T^{- 1}$, as it is for the stochastic differential equation (\ref{eq: model}) without jumps. For this, we need the following additional assumption on {\modch $a$}.
\begin{enumerate}
\item[{\bf A6}] {\modch If $d=1$, $a^2 \in C^2_b(\R)$, that is, $a^2$ is twice continuously differentiable with bounded first and second derivatives.} 
\end{enumerate}
{\modch Assumption \textbf{A6} is needed in order to show the results gathered in Theorems \ref{th: upper bound} and \ref{th: as norm}, while for the other results only assumptions \textbf{A1} - \textbf{A5} will be required.}
\begin{theorem} \label{th: upper bound}
Let $X$ be the solution to \textnormal{(\ref{eq: model})} on $[0,T]$ with $d=1$.
Suppose that Assumptions {\bf A1}-{\bf A6} hold and $\mu \in \mathcal{H}_1(\beta, \mathcal{L})$, {\modch with $\beta \ge 1$}. Then  there exists $c>0$ such that for all $T>0$, $h \leq 1$, and $x \in \R$,
\begin{equation} \label{pre}
\mathbb{E} [|\hat{\mu}_{h,T}(x) - \mu(x)|^2] \leq ce^{\epsilon |x|}( h^{2 \beta} + \frac{1}{T}),
\end{equation}
where $0<\epsilon \le \min(\frac{\epsilon_0}{\left \| \gamma \right \|_\infty}, \epsilon_0)$, with $\epsilon_0 > 0$ as in Assumption {\bf A4}
In particular, choosing $h(T) = \frac{1}{\sqrt{T}}$, we conclude that
{\modch for $T \geq 1$},
\begin{equation} \label{1t}
\mathbb{E} [|\hat{\mu}_{h,T}(x) - \mu(x)|^2] \leq  \frac{c e^{\epsilon |x|}}{T}.
\end{equation}
\end{theorem}
{\modch We observe that both the bandwidth and the upper bound do not depend on the unknown smoothness of the invariant density $\beta$, so there is no need to propose a data driven bandwidth adaptive selection procedure as in the case $d>2$ (see \cite{Chapitre 4}).} 

Theorem \ref{th: upper bound} improves the upper bound  obtained in  \cite{Chapitre 4} which was of the form $\frac{(\log  T)^{(2-\frac{1+\alpha}{2}) \vee 1}}{T}$. The price to pay is that the constant in the upper bound depends on $x$ (see Remark \ref{rem} below). However, we are able to find a convergence rate which is optimal, as we will see in Theorem \ref{th: lower bound}. As in \cite{Chapitre 4}, we will use the bias-variance decomposition  (see \cite[Proposition 1]{Decomp})
\begin{equation} \label{bv}\begin{split}
\mathbb{E} [|\hat{\mu}_{h,T}(x) - \mu (x)|^2] &\leq {\modch \vert\mathbb{E} [\hat{\mu}_{h,T}(x)] - \mu (x)|^2
+\mathbb{E} [|\hat{\mu}_{h,T}(x) - \mathbb{E}[\hat{\mu}_{h,T}(x)]|^2]} \\
&\leq c\left( h^{2\beta} + T^{- 2}  \text{Var} \left( \int_0^T \mathbb{K}(x - X_t) dt\right)\right),
\end{split}
\end{equation}
{\modch for some constant $c>0$. For the proof of the bias bound $c h^{2\beta}$ in the same setting of this paper see the proof of Proposition 2 in \cite{Chapitre 4}.}

Then in \cite{Chapitre 4} bounds on the transition semigroup and on the transition density (see (\ref{pt}) above) give an upper bound for the variance depending on the bandwidth.
Here, we use a similar approach as in \cite{Cas_Lea} and \cite{Comte_Mer} to obtain a bandwidth-free rate for the variance of smoothing density estimators (which include the kernel estimator). 
 For Markov diffusions, the sufficient conditions can be decomposed into a local irregularity condition {\bf WCL1} plus an asymptotic independence condition {\bf WCL2}.
There exist two positive integrable functions $k_1$ and $k_2$ (defined on $\mathbb{R}$) and $u_0>0$ such that
\begin{align*}
&\mbox{\bf WCL1: } \sup_{y \in \mathbb{R}} \int_0^{u_0} |g_u(x,y)|\, du < k_1(x), \qquad \text{for all } x \in \mathbb{R},\\
&\mbox{\bf WCL2: } \sup_{y \in \mathbb{R}} \int_{u_0}^\infty  |g_u(x,y)| \, du < k_2(x), \qquad \text{for all } x \in \mathbb{R}
\end{align*}
where $g_u(x,y):=\mu(x) p_u(x,y) - \mu(x) \mu(y)$.
In order to show these conditions, {\modch some further  bounds on the transition density $p_t(x,y)$ involving partial derivatives are needed} (see Lemma \ref{lemma: bound second derivatives p} below), for which the  additional condition {\bf A6} is required.

\begin{remark} \label{rem}
The term $e^{\epsilon \vert  x \vert}$ that appears in the bounds \textnormal{(\ref{pre})} and \textnormal{(\ref{1t})} comes from the fact that we are able to show condition {\bf  WCL2} with $k_2(x)=\mu(x)(1+f^{\ast}(x))$, where $f^{\ast}$ is the  Lyapunov function constructed in \textnormal{\cite{Chapitre 4}}, defined as a $C^{\infty}$ approximation of $e^{\epsilon \vert  x \vert}$ (see the proof of Proposition \ref{prop: WCL satisfied}). We know that $\int_{\R} \mu(x) f^{\ast}(x) dx <\infty$, as shown in \textnormal{\cite{Masuda}}, but this is not sufficient as it was in \textnormal{\cite{Comte_Mer}} in order to bound the variance term in \textnormal{(\ref{bv})} since here  we are dealing with the  kernel estimator. In order to remove the term $e^{\epsilon \vert  x \vert}$ an additional assumption would be  needed that ensures that $\sup_{x \in \R} \mu(x) f^{\ast}(x)<\infty$.
\end{remark}

As shown in \cite{As norm}, conditions {\bf WLC1} and {\bf WLC2} are also useful to show the asymptotic normality of the kernel density estimator, as proved in the  next theorem.
\begin{theorem} \label{th: as norm}
Let $X$ be the solution to \textnormal{(\ref{eq: model})} on $[0,T]$ with $d=1$.
Suppose that Assumptions {\bf A1}-{\bf A6} hold and $\mu \in \mathcal{H}_1(\beta, \mathcal{L})$, {\modch with $\beta \geq  1$. Consider the bandwidth $h(T)= (\frac{1}{T})^{\frac{1}{2}- \epsilon}$, where $\epsilon \in (0, \frac12)$.} Then,
for any collection $(x_i)_{1 \le i \le m}$ of distinct real numbers
\begin{equation} \label{eq: as norm 1}
\sqrt{T}(\hat{\mu}_{h, T} (x_i) - \mathbb{E}[\hat{\mu}_{h, T} (x_i)], \, 1 \le i \le m) \xrightarrow{\mathcal{D}} N^{(m)} (0, \Sigma^{(m)}) \text{ as } T \rightarrow \infty,
\end{equation}
where
$$\Sigma^{(m)} := (\sigma(x_i, x_j))_{1 \le i, j \le m}, \qquad \sigma(x_i, x_j) := 2 \int_0^{\infty} g_u(x_i, x_j) du.$$
\end{theorem}

{\modch Observe that using the choice of $h(T)= (\frac{1}{T})^{\frac{1}{2}- \epsilon}$, with $\epsilon>0$ in the bias bound (\ref{bv}),
we get that for any $x \in \R$ and $T \geq 1$,
$$
\sqrt{T} \vert \mathbb{E} [\hat{\mu}_{h,T}(x) ] -\mu(x) \vert \leq c T^{-\frac12(\beta-1-2\beta \epsilon)}.
$$
Therefore, choosing $\beta>1$ and $\epsilon<\frac{\beta-1}{2\beta}$ and applying Theorem \ref{th: as norm}, we conclude that
as $T \rightarrow \infty$
\begin{equation*} 
\sqrt{T}(\hat{\mu}_{h, T} (x_i) - \mu (x_i), \, 1 \le i \le m) \xrightarrow{\mathcal{D}} N^{(m)} (0, \Sigma^{(m)}).
\end{equation*}}

We are also interested in obtaining lower bounds in dimension $d \in \{1,2\}$.
{\modch For the computations of the lower bounds we consider the particular case of equation (\ref{eq: model}) given by}
\begin{equation}\label{eq: model lower bound}
X_t= X_0 + \int_0^t b( X_s)ds + aB_t + \int_0^t \int_{\mathbb{R}^d_0}
\gamma z (\nu(ds,dz)-F(z) dz ds),
\end{equation}
where $a$ and $\gamma$ are $d \times d$ constant matrices and the rest of terms are as in equation (\ref{eq: model}).

{\modch We next introduce the following set of drift functions of equation (\ref{eq: model lower bound}).}
We say that a bounded and Lipschitz function  $b: \R^d \rightarrow \R^d$  belongs to $\Sigma (\beta, \mathcal{L})$ if the unique invariant density $\mu_b$ of  the solution $X=(X_t)_{ t \geq 0}$ to (\ref{eq: model lower bound}) belongs to $\mathcal{H}_d ( \beta, {\modch 2} \mathcal{L})$ for some $\beta, \mathcal{L} \in \R^d$, $\beta_i\geq 1$, $\mathcal{L}_i>0$. {\modch A detailed description of the set $\Sigma (\beta, \mathcal{L})$ will be given in Section 4.3, where two explicit examples of drift coefficients $b_0$ and $b_1$ belonging to $\Sigma (\beta, \mathcal{L})$ will be introduced.}

We denote by $\mathbb{P}_b^{(T)}$ and $\mathbb{E}_b^{(T)}$ the law and expectation of the solution $(X_t)_{t \in [0, T]}$.
We define the minimax risk at a point $x \in \mathbb{R}^d$ by
\begin{equation}\label{eq: def risk}
\mathcal{R}^x_T (\beta, \mathcal{L}) := \inf_{\tilde{\mu}_T} \mathcal{R} (\tilde{\mu}_T (x)):=\inf_{\tilde{\mu}_T}\sup_{b \in \Sigma (\beta, \mathcal{L})} \mathbb{E}_b^{(T)}[(\tilde{\mu}_T (x) - \mu_b (x))^2],
\end{equation}
where the infimum is taken on all possible estimators of the invariant density.

The following lower bounds hold true. 
\begin{theorem} \label{th: lower bound}
Let $X$ be the solution to \textnormal{(\ref{eq: model lower bound})} on $[0,T]$ with $d=1$.
{\modch Suppose that Assumptions {\bf A1}-{\bf A5} hold, that $\int_{\R}F(z) dz < \infty$ and that $\mu_b \in \mathcal{H}_1 (\beta, \mathcal{L})$, with $\beta \geq 1$.}
Then, there exists $T_0 > 0$ and $c>0$ such that, for all $T \ge T_0$,
$$\inf_{x \in \R}\mathcal{R}^x_T (\beta, \mathcal{L})\geq \frac{c}{T}.$$
\end{theorem}

\begin{theorem} \label{th: lower d=2}
Let $X$ be the solution to \textnormal{(\ref{eq: model lower bound})} on $[0,T]$ with $d=2$. {\modch Suppose that Assumptions {\bf A1}-{\bf A5} hold, that $\int_{\R^2}F(z) dz < \infty$ and that $\mu_b \in \mathcal{H}_2 (\beta, \mathcal{L})$, with $\beta_i \geq  1$ for $i =1, 2$.}
Assume that for all $i \in  \{1,2\}$ and $j \neq i$, 
\begin{equation} \label{conditiona}
\vert (a  a^T)_{ij} (a a^T)^{-1}_{jj} \vert \leq \frac12.
\end{equation}
Then, there exists $T_0 > 0$ and $c>0$ such that, for $T \ge T_0$,
$$\inf_{\tilde{\mu}_T} \sup_{b \in \Sigma(\beta, \mathcal{L})} \mathbb{E}^{(T)}_b\left[ \sup_{x \in \R^2} (\tilde{\mu}_T(x) - \mu_b(x))^2\right] \geq c \frac{\log T}{T}.$$
\end{theorem}

{\modch Recall for these two theorems, $a$ and $\gamma$ are $d \times d$ constant matrices. In this case, when $d=1$, Assumption {\bf A1} is equivalent to say that $a\neq 0 $ and $\gamma>0$, while when $d=2$, it is equivalent to say that $\det(a) \neq 0$ and $\det(\gamma)>0$. 
Moreover, hypotheses {\bf A3}-{\bf A5} imply that the unique solution to equation (\ref{eq: model lower bound}) admits a unique invariant measure $\pi_b$, which we assume has a density $\mu_b$ with respect to the Lebesgue measure, as before.}

Comparing these lower bounds  with the upper bound of  Theorem \ref{th: upper bound} for the case  $d=1$ and Proposition 4 in \cite{Chapitre 4} for the two-dimensional case, we conclude that the convergence rate $\{\frac{1}{T},\frac{\log T}{T}\}$ are the best possible for the {\modch estimation} of the invariant density in dimension $d \in \{1,2\}$.

The proof of Theorem \ref{th: lower bound} follows along the same lines as that of Theorem 2 in \cite{lowerbound}, where a lower bound for the {\modch estimation} of the invariant  density for the solution to (\ref{eq: model lower bound}) for $d \geq 3$ is obtained. The proof is based on the two hypotheses  method, explained for example in Section 2.3 of \cite{Ts}. However, this method does not work for the two-dimensional case as explained in  Remark \ref{d12} below. Instead, we use the Kullback's version of the finite number of hypotheses method as  stated in Lemma C.1 of \cite{Strauch}, see Lemma \ref{lemma: lower dim 2} below. Observe that this method gives a slightly weaker lower bound as we get a $\sup_x$ inside the  expectation, while
the method in \cite{lowerbound} provides an $\inf_x$ outside the expectation.

\section{Preliminary results}{\label{S: prel}}

The proof of Theorems \ref{th: upper bound} and \ref{th: as norm} will use the following {\modch bounds on the transition density.
 \begin{lemma}  \label{lemma: bound second derivatives p}
Let $X$ be the solution to \textnormal{(\ref{eq: model})} on $[0,T]$ with $d=1$. Suppose that Assumptions {\bf  A1-A6} hold. 
Then, there exist jointly continuous processes $Z$, $A$ and $B$ on $\R_+\times \R^2$ such that
for all $t\geq 0$ and $x,y \in \R$,
\begin{equation} \label{a0}
p_t(x,y)=Z_t(x,y)+A_t(x,y)+B_t(x,y)
\end{equation}
satisfying that
for all $T>0$, there exist  $c>0$ and $\lambda> 0$ such that for any $x, y \in \mathbb{R}$ and $t \in [0, T]$
 \begin{equation} \label{a1}
\bigg| \frac{\partial^2}{\partial y^2} Z_t(x,y) \bigg|\le c\,  t^{-3/2} e^{-\lambda \frac{|y-x|^2}{t}},
 \end{equation}
 \begin{equation} \label{a3}
 \vert A_t(x,y) \vert \leq c\,  (t^{3/2} (\vert y-x\vert+\sqrt{t})^{-1-\alpha}+e^{-\lambda \frac{|y-z|^2}{t}}),
 \end{equation}
 and
 \begin{equation} \label{a4}
 \vert B_t(x,y) \vert \leq c \, (1+ t^{2-\alpha/2})(\vert y-x\vert+\sqrt{t})^{-1-\alpha}.
 \end{equation}
 \end{lemma}
 
 \begin{proof}
By Duhamel's formula (1.12) of \cite{Chen},  the transition density of the solution to \textnormal{(\ref{eq: model})} satisfies that for all $t\geq 0$ and $x,y \in \R$,
\begin{equation*} 
p_t(x,y)=Z_t(x,y)+A_t(x,y)+B_t(x,y)
\end{equation*}
where $Z_t(x,y)$ is the transition density of the solution to \textnormal{(\ref{eq: model})} with $b =\gamma=0$, and $A_t$ and $B_t$ are defined as follows
$$
A_t(x,y):=\int_0^t \int_{\R} p_r(x,z) \,  b(z) \frac{\partial}{\partial z} Z_{t-r}(z,y) \, dz \, dr,
$$
and
\begin{equation*} \begin{split}
B_t(x,y)&:=\int_0^t \int_{\R} p_r(x,z) \int_{\R} \big(Z_{t-r}(z+\xi,y)-Z_{t-r}(z,y)\\
&\qquad \qquad \qquad -{\bf 1}_{\vert \xi \vert \leq 1} \, \xi\,  \frac{\partial}{\partial z} Z_{t-r}(z,y) \big) \frac{k(z,\xi)}{\vert \xi \vert^{1+\alpha}} \, d\xi \, dz \, dr,
\end{split}
\end{equation*}
where $k(z,\xi)=\frac{1}{\gamma(z)}\vert \xi \vert^{1+\alpha} F(\frac{\xi}{\gamma(z)})$. This shows the decomposition formula (\ref{a0}).

By (6.1) in Theorem 7 of \cite{friedman}, using the fact that $a^2$ is bounded together with {\bf A6}, we have that for all $T>0$, there exist  $c, \lambda> 0$ such that for all $x, y \in \mathbb{R}$ and $t \in [0, T]$
 \begin{equation*} 
\bigg| \frac{\partial^2}{\partial y^2} Z_t(x,y) \bigg|\le c \,  t^{-3/2} e^{-\lambda \frac{|y-x|^2}{t}},
 \end{equation*}
 (which proves (\ref{a1})) and
 \begin{equation} \label{a2}
\bigg| \frac{\partial}{\partial x} Z_t(x,y) \bigg|\le c \,  t^{-1} e^{-\lambda \frac{|y-x|^2}{t}}.
 \end{equation}
 In particular, using (\ref{a2}) and the fact that $b$ is bounded, we get that
 \begin{equation*}
 \vert A_t(x,y) \vert \leq c \int_0^t \int_{\R} p_r(x,z) (t-r)^{-1} e^{-\lambda\frac{|y-z|^2}{t-r}} \, dz \, dr.
 \end{equation*}
 Moreover, using (\ref{pt}) together with (2.6) and (2.8) of \cite{Chen} with $\gamma_1=-1$ and $\gamma_2=2$, and $\gamma_1=0$ and $\gamma_2=-1$, respectively, we conclude that (\ref{a3}) holds true.
  
 On the other hand, appealing to Corollary 2.4(i) of \cite{Chen}, from hypotheses {\bf A1}, {\bf A3} and {\bf A5}, we get that for all $T>0$, there exists $c> 0$ such that for all $x, y \in \mathbb{R}$ and $t \in [0, T]$,
 \begin{equation*} 
\vert  B_t(x,y) \vert \leq c \int_0^t \int_{\R} p_r(x,z)(\vert y-z \vert +\sqrt{t-r})^{-1-\alpha} \, dz \, dr.
 \end{equation*}
 Finally, 
 using again (\ref{pt}) together with (2.5) and (2.6) of \cite{Chen} with $\gamma_1=0$ and $\gamma_2=2$, and $\gamma_1=0$ and $\gamma_2=0$, respectively, we obtain (\ref{a4}).

  The proof of the Lemma is completed.
\end{proof}}
 
The key point of the proof of Theorem \ref{th: upper bound} consists in showing that conditions {\bf WCL1} and {\bf WCL2} hold true, which is proved in the next proposition.
\begin{proposition} \label{prop: WCL satisfied}
Let $X$ be the solution to \textnormal{(\ref{eq: model})} on $[0,T]$ with $d=1$.
Suppose that Assumptions {\bf A1-A6} hold. Then, conditions {\bf WCL1} and {\bf WCL2} are satisfied.
\end{proposition}
\begin{proof}
We start considering {\bf WCL1}. The density estimate (\ref{pt}) yields 
\begin{equation}
p_t(x,y) \le c t^{- \frac{1}{2}} + \tilde{c} t^{\frac{1 - \alpha}{2}} \le \bar{c} t^{- \frac{1}{2}} \qquad 0 < t \le 2,
\label{eq: bound transition density}
\end{equation}
which combined with $\sup_{y \in \mathbb{R}} \mu(y) < \infty$ gives {\bf WCL1} with $k_1 (x) = \mu(x)$ and $u_0=2$. In order to show {\bf WCL2}, we set $\varphi(\xi) := \mathbb{E}[\exp(i \xi X_t)]$ and $\varphi_x(\xi, t) := \mathbb{E}[\exp(i \xi X_t)| X_0 = x]$ and we claim that there exists  $\hat{c} > 0$ such that for all $\xi \in \mathbb{R}$,
    \begin{equation} 
    |\varphi(\xi)| \le {\modch \hat{c}}(1 + |\xi|)^{-2}.
    \label{eq: cond varphi}
    \end{equation}
    Moreover, there exists $\tilde{c} > 0$,  such that for all $t \geq 2$, $x \in \mathbb{R}$, and $\xi \in \mathbb{R}$,
     \begin{equation}
    |\varphi_x(\xi, t)| \le {\modch \tilde{c}}(1 + |\xi|)^{-2}.
        \label{eq: cond varphi x}
    \end{equation}
Recall from Lemma 2 in \cite{Chapitre 4} and its proof that the process $X$ is exponentially $\beta$-mixing and there exists $\rho>0$ such that for all $x \in \mathbb{R}$ and $t>0$,
\begin{equation}
\left \| P_t(x, \cdot ) - \mu(\cdot) \right \|_{TV} \le (1 + f^*(x))e^{- \rho t},
\label{eq: beta mixing}
\end{equation}
where $(P_t)_{t \in \R}$ is the transition semigroup of our process $X$, $\left \| \cdot \right \|_{TV}$ is the total variation norm and $f^*(x)$ is a Lyapounov function. Specifically, $f^*(x)$ is defined as $e^{\epsilon |x|}$ for $|x| \ge 1$, with $\epsilon \le \min (\frac{\epsilon_0}{\left \| \gamma \right \|_\infty}, \epsilon_0)$ ($\epsilon_0 > 0$ as in Assumption {\bf A4}). In order to avoid any regularity problem in $0$, $f^*$ is introduced as piecewise function. For $|x| < 1$ it is defined as a $C^\infty$ approximation of $e^{\epsilon |x|}$, such that $f^*$ is $C^\infty$ on $\mathbb{R}$. 

We now prove that inequalities \eqref{eq: cond varphi}, \eqref{eq: cond varphi x} and \eqref{eq: beta mixing} imply {\bf WCL2}. 
Using the inverse Fourier transform, we have 
$$2 \pi (p_t(x,y) - \mu(y)) = \int_{\mathbb{R}} \exp(- i \xi y)(\varphi_x(\xi, t)-  \varphi(t)) d\xi.$$
Then, using \eqref{eq: cond varphi} and \eqref{eq: cond varphi x} we get, for $t\ge 2$, 
$$2 \pi \vert p_t(x,y) - \mu(y)\vert \le 2(\tilde{c} + \hat{c})^{\frac{p-1}{p}} (\sup_{\xi \in \mathbb{R}}|\varphi_x(\xi, t)-  \varphi(\xi)|)^{\frac{1}{p}} \int_{\mathbb{R}^+}(1 + \xi)^{- 2 \frac{p-1}{p}} d\xi,$$
where we have used that $1 = \frac{1}{p} + \frac{p - 1}{p}$. We can choose $p > 2$, so that $2 \frac{p-1}{p} > 1$. We get that there exists a finite constant $c$ such that, for all $t \ge 2$ and $x, y \in \R$,
$$\vert g_t(x,y) \vert=\mu(x)|p_t(x,y) - \mu(y)| \le c \mu(x)(\sup_{\xi \in \mathbb{R}}|\varphi_x(\xi, t)-  \varphi(\xi)|)^{\frac{1}{p}},$$
where  we observe that the right hand side is independent of $y$.
By using the fact that 
$$\sup_{\lambda \in \mathbb{R}}|\varphi_x(\lambda, t)-  \varphi(\lambda)|\leq  \left \| P_t(x, \cdot ) - \mu(\cdot) \right \|_{TV}$$
together with \eqref{eq: beta mixing} we obtain that there exist $c>0$ and $\rho>0$ such that for all $x,y\in \R$ and $t \geq  2$,
\begin{align*}
|g_t(x,y)| & \le c \mu(x)(1 + f^*(x)) e^{- \rho t},
\end{align*}
as $f^{\ast}$ is positive,
and so 
$$ \sup_{y \in \mathbb{R}} \int_{2}^{\infty} \vert g_t(x,y) \vert \, dt \le c \mu(x)(1 + f^*(x)) \int_{2}^{\infty}  e^{- \rho t} \,dt, $$
which implies \textbf{WCL2} with $k_2(x) = c \mu(x)(1 + f^*(x))$.

 We are left to show \eqref{eq: cond varphi} and \eqref{eq: cond varphi x}.
We start showing \eqref{eq: cond varphi x}. Using (\ref{a0}) and integrating by parts yields
\begin{equation*} \begin{split}
 |\varphi_x(\xi, t)| &= \bigg|\int_{\mathbb{R}} \exp(i \xi y) p_t(x,y) dy\bigg| \\
 &=\bigg|\int_{\mathbb{R}}  \int_{\mathbb{R}}\exp(i \xi y) p_{t-1}(x,z) p_1(z,y) dy \, dz \bigg|\\ 
  &=\bigg|\int_{\mathbb{R}}  \int_{\mathbb{R}}\exp(i \xi y) p_{t-1}(x,z) \left(Z_1(z,y)+A_1(z,y)+B_1(z,y)\right) dy\,  dz\bigg|\\
 &\leq  |\xi|^{-2} \int_{\mathbb{R}}\int_{\mathbb{R}} p_{t-1}(x,z) \bigg|\frac{\partial^2}{\partial y^2} Z_1(z,y)\bigg| dy \, dz\\
 &\qquad 
 +\int_{\mathbb{R}}  \int_{\mathbb{R}} p_{t-1}(x,z) \vert A_1(z,y)\vert  dy \, dz
 +\int_{\mathbb{R}}  \int_{\mathbb{R}} p_{t-1}(x,z) \vert B_1(z,y)\vert  dy \, dz\\
 &=:|\xi|^{-2}(I_1+I_2+I_3).
 \end{split}
\end{equation*}
Appealing to (\ref{a1}), we obtain that
\begin{equation*} \begin{split}
I_1\le c \int_{\mathbb{R}}\int_{\mathbb{R}}p_{t-1}(x,z) e^{-\lambda|y-z|^2} dy\, dz =c,
\end{split}
\end{equation*}
where $c$ is independent of $t$ and $x$ as $\int_{\mathbb{R}}p_{t-1}(x,z)dz=1$.
Using (\ref{a3}), we get that
\begin{equation*} \begin{split}
I_2 &\leq c\int_{\mathbb{R}}  \int_{\mathbb{R}} p_{t-1}(x,z) (\vert y-z\vert+1)^{-1-\alpha}+e^{-\lambda|y-z|^2})dy \, dz=c,
\end{split}
\end{equation*}
as the $dy$ integral is finite since $\alpha \in (0,2)$.
Similarly, by (\ref{a4}), 
\begin{equation*} \begin{split}
&I_3 \leq c\int_{\mathbb{R}}  \int_{\mathbb{R}} p_{t-1}(x,z) (\vert y-x\vert+1)^{-1-\alpha} dy \, dz \leq c.
\end{split}
\end{equation*}
Thus, we have proved that
$
|\varphi_x(\xi, t)|\leq c \vert \xi \vert^{-2}$.
Since $|\varphi_x(\xi, t)|\leq 1$, this implies
\eqref{eq: cond varphi x}. Similarly,
\begin{equation*} \begin{split}
|\varphi(\xi)| &= \bigg|\int_{\mathbb{R}} \exp(i \xi y) \mu(y) dy\bigg| \\
& \leq \bigg|\int_{\mathbb{R}} \int_{\mathbb{R}}\exp(i \xi y)  \mu(z) Z_1(z,y)  dy\, dz \bigg| \\
 &\qquad 
 +\int_{\mathbb{R}}  \int_{\mathbb{R}} \mu(z) \vert A_1(z,y)\vert  dy \, dz
 +\int_{\mathbb{R}}  \int_{\mathbb{R}} \mu(z) \vert B_1(z,y)\vert  dy \, dz
\\
&\le |\xi|^{-2}\int_{\mathbb{R}}\int_{\mathbb{R}} \mu(z) \bigg|\frac{\partial^2}{\partial y^2} Z_1(z,y) \bigg| dy\, dz \\
&\qquad 
 +\int_{\mathbb{R}}  \int_{\mathbb{R}} \mu(z) \vert A_1(z,y)\vert  dy \, dz
 +\int_{\mathbb{R}}  \int_{\mathbb{R}} \mu(z) \vert B_1(z,y)\vert  dy \, dz\\
&\le 
c |\xi|^{-2},
\end{split}
\end{equation*}
which implies \eqref{eq: cond varphi x} since $\vert \varphi (\xi) \vert \leq 1$.  The proof of the proposition is now completed.
\end{proof}

Theorem \ref{th: as norm} is an application of the following central limit theorem for discrete stationary sequences.
Let $Y_n = (Y_{n,i}, \, i \in \mathbb{Z})$, $ n \ge 1$ be a sequence of strictly stationary discrete time $\mathbb{R}^m$ valued random process. We define the $\alpha$-mixing coefficient of $Y_n$ by
$$\alpha_{n,k} := \sup_{A \in \sigma(Y_{n,i}, \, i \le 0), \quad B \in \sigma(Y_{n,i}, \, i \ge k)}{\modch \big(\mathbb{P}(A \cap B) - \mathbb{P}(A) \mathbb{P}(B) \big) }$$
and we set $\alpha_k := \sup_{n \ge 1} \alpha_{n,k}$ (see also Section 1 in \cite{Mixing}). We denote by $Y^{(r)}$ the r-th component of an $m$ dimensional random vector $Y$. 
\begin{theorem} [Theorem 1.1 \cite{As norm}] \label{th: teorema 1.1}Assume that 
\begin{enumerate}
\item[\textnormal{(i)}] $\mathbb{E}[Y_{n,i}^{(r)}] = 0$ and $|Y_{n,i}^{(r)}| \le M_n$ for every $n \ge 1$, $i \ge 1$ and $1 \le r \le m$, where $M_n$ is a constant depending only {\modch on} $n$. 
\item[\textnormal{(ii)}] $$\sup_{i \ge 1, 1 \le r \le m} \mathbb{E}[(Y_{n,i}^{(r)})^2] < \infty.$$ 
\item[\textnormal{(iii)}] For every $1 \le r, s \le m$ and for every sequence $b_n \rightarrow \infty$ such that $b_n \le n$ for every $n \ge 1$, we have 
$$\lim_{n \rightarrow \infty} \frac{1}{b_n} \mathbb{E}\left[\sum_{i = 1}^{b_n} Y_{n,i}^{(r)} \sum_{j = 1}^{b_n} Y_{n,j}^{(s)} \right] = \sigma_{r,s}.$$
\item[\textnormal{(iv)}] There exists ${\modch \gamma_0} \in (1, \infty)$ such that $\sum_{k \ge 1} k \alpha_k^{\frac{{\modch \gamma_0}  - 1}{{\modch \gamma_0} }} < \infty$. 
\item [\textnormal{(v)}] For some constant $c > 0$ and for every $n \ge 1$, $M_n \le c n^{\frac{{\modch \gamma_0} ^2}{(3 {\modch \gamma_0} - 1)(2 {\modch \gamma_0}  - 1)}}$. 
\end{enumerate}
Then, 
$$\frac{\sum_{i = 1}^n Y_{n,i}}{\sqrt{n}} \xrightarrow{\mathcal{D}} N(0, \Sigma) \quad \mbox{as } n \rightarrow \infty,$$
where $\Sigma= (\sigma_{r,s})_{1 \le r, s \le m}$.
\end{theorem}
The proof of Theorem \ref{th: lower d=2} is based on the following Kullback version of
the main theorem on lower bounds in \cite{Ts}, see Lemma C.1 of \cite{Strauch}:
\begin{lemma} \label{lemma: lower dim 2}
Fix $\beta, \mathcal{L} \in (0, \infty)^2$ and assume that there exists $f_0 \in \mathcal{H}_2(\beta, \mathcal{L}) $ and a finite set $J_T$ such that one can find $\left \{ f_j, \, j \in J_T \right \} \subset \mathcal{H}_2(\beta, \mathcal{L})$ satisfying
\begin{equation} \label{f_diference}
\left \| f_j - f_k \right \|_\infty \ge 2 \psi > 0 \qquad \forall j \neq k \in J_T.
\end{equation}
Moreover, denoting {\modch $\mathbb{P}_j^{(T)}$} the probability measure associated with $f_j$, $\forall j \in J_T$, $\mathbb{P}^{(T)}_j \ll \mathbb{P}^{(T)}_0$ and 
\begin{equation} \label{step2}
\frac{1}{|J_T|} \sum_{j \in J_T} KL (\mathbb{P}^{(T)}_j, \mathbb{P}^{(T)}_0) = \frac{1}{|J_T|} \sum_{j \in J_T} \mathbb{E}^{(T)}_j\left[\log\left(\frac{d \mathbb{P}^{(T)}_j}{d \mathbb{P}^{(T)}_0}(X^T)\right)\right] \le {\modch \delta} \log (|J_T|)
\end{equation}
for some ${\modch \delta} \in (0, \frac{1}{8})$. Then, for $q > 0$, we have 
$$\inf_{\tilde{\mu}_T} \sup_{\mu_b \in \mathcal{H}_2(\beta, \mathcal{L})} (\mathbb{E}^{(T)}_b[\psi^{- q} \left \| \tilde{\mu}_T - \mu_b \right \|_\infty^q])^{1/q} \ge c({\modch \delta}) > 0,$$
where the infimum is taken over all the possible estimators $\tilde{\mu}_T$ of $\mu_b$.
\end{lemma}

\section{Proof of the main results}{\label{S: proof main}}

\subsection{Proof of Theorem \ref{th: upper bound}}

 By the symmetry of the covariance operator and the stationarity of the process, 
 \begin{equation*} \label{eq: upper bound variance}\begin{split}
&T \, \text{Var}(\hat{\mu}_{h,T}(x)) = \frac{1}{T}\int_0^T \int_0^T \text{Cov}(\mathbb{K}_h (x - X_t), \mathbb{K}_h (x - X_s)) ds\, dt\\
&\qquad =
 \frac{2}{T} \int_0^T (T-u) \text{Cov}(\mathbb{K}_h (x - X_u), \mathbb{K}_h (x - X_0)) du\\
&\qquad =2 \int_0^T (1 - \frac{u}{T}) \int_{\mathbb{R}} \int_{\mathbb{R}} \mathbb{K}_h (x - y) \mathbb{K}_h (x - z) g_u(y,z) dy\, dz\, du \\
 &\qquad \le \int_{\mathbb{R}} \vert \mathbb{K}_h (x - y) \vert \sup_{z \in \mathbb{R}}\int_0^\infty \vert g_u(y,z)\vert  du \, dy \int_{\mathbb{R}} \vert \mathbb{K}_h (x - z)\vert dz.
\end{split}
\end{equation*}
In the proof of Proposition \ref{prop: WCL satisfied} we have shown that
$$\sup_{z \in \mathbb{R}}\int_0^\infty \vert g_u(y,z)\vert  du \leq c(1+\mu(y)(1+ f^*(y))).$$
It follows that
$$T \, \text{Var}(\hat{\mu}_{h,T}(x)) \le c \int_{\mathbb{R}}\vert \mathbb{K}_h (x - y)\vert (1+\mu(y)(1+ f^*(y))) dy,$$
since, by the definition of the kernel function, 
$$\int_{\R} \vert \mathbb{K}_h (x - z) \vert dz  = \int_{x - h}^{x + h} \vert \mathbb{K}_h (x - z) \vert dz \le \left \| \mathbb{K}_h \right \|_\infty h \le \frac{\left \| K \right \|_\infty}{h} h = \left \| K \right \|_\infty.$$ 
Then, by the definition of $\mathbb{K}_h$, we get that
\begin{equation*} \begin{split}
&\int_{\mathbb{R}}\vert \mathbb{K}_h (x - y)\vert (1+\mu(y)(1+ f^*(y))) dy\\
&\qquad \qquad =\frac{1}{h} \int_{x - h}^{x + h} \vert K(\frac{x - y}{h})\vert  (1+\mu(y)(1+ f^*(y))) dy \\
&\qquad \qquad \leq \left \| K \right \|_\infty \int_{-1}^1   (1+ \mu(x-h\tilde{y})(1+ f^*(x - h \tilde{y}))) d\tilde{y},
\end{split}
\end{equation*}
where we have applied the change of variable $\tilde{y}:= \frac{x - y}{h}$. Now we observe that, if $|x - h \tilde{y}| \le 1$, then $f^*(x - h \tilde{y})$ is bounded by construction. Otherwise, for $|x - h \tilde{y}| > 1$, we have 
\begin{align*}
f^*(x - h \tilde{y}) & = e^{\epsilon |x - h \tilde{y}|} \le e^{\epsilon |x|} e^{ \epsilon h |\tilde{y}|} \le e^{\epsilon |x|} e^{ \epsilon}, 
\end{align*}
where in the last inequality we have used the fact that both $h$ and $|\tilde{y}|$ are smaller than $1$.  Therefore, we have shown that
$$T \, \text{Var}(\hat{\mu}_{h,T}(x)) \le c e^{\epsilon |x|},$$
where $c$ is independent of $T$, $h$ and $x$.
Finally, from the bias-variance decomposition (\ref{bv}) we obtain (\ref{pre}), which concludes the desired proof.

\subsection{Proof of Theorem \ref{th: as norm}}
 
We aim to apply Theorem \ref{th: teorema 1.1}. {\modch For this, we split the interval $[0, T]$ into $n$ intervals $[t_{i-1}, t_i]$,  where $t_i = i \Delta$ for any $i \in \left \{0,\ldots, n \right \}$, $n \Delta=T$, and $n = \lfloor T \rfloor$ with $T \ge 1$, which implies that $1 \le \Delta < 2$.}
 
For each $n \ge 1$ and  $1 \le r \le m$, we consider the sequence $(Y_{n,i}^{(r)})_{i \ge 1}$ defined as 
$$Y_{n,i}^{(r)} := \frac{1}{\sqrt{\Delta}}\left(\int_{t_{i - 1}}^{t_i} \mathbb{K}_h (x_r - X_u) du - \mathbb{E}\left[\int_{t_{i - 1}}^{t_i} \mathbb{K}_h (x_r - X_u) du\right]\right),$$
for $x_r \in \mathbb{R}$. We denote  by $Y_{n,i}$ the $\mathbb{R}^m$ valued random vector defined by $Y_{n,i} = (Y_{n,i}^{(1)},\ldots, Y_{n,i}^{(m)})$. By construction, 
$$\frac{\sum_{i = 1}^n Y_{n,i}}{\sqrt{n}}= \sqrt{T} (\hat{\mu}_{h, T} (x) - \mathbb{E}[\hat{\mu}_{h, T} (x)]),$$
where $\hat{\mu}_{h, T} (x) - \mathbb{E}[\hat{\mu}_{h, T} (x)]$ is the vector $$(\hat{\mu}_{h, T} (x_1) - \mathbb{E}[\hat{\mu}_{h, T} (x_1)],\ldots, \hat{\mu}_{h, T} (x_m) - \mathbb{E}[\hat{\mu}_{h, T} (x_m)]).$$

It is clear that $\mathbb{E}[Y_{n,i}] = 0$ for all $n \ge 1$ and $i \ge 1$. Moreover, for all $i \ge 1$, $1 \le r \le m$ and $n \ge 1$ we have 
$$|Y_{n,i}^{(r)}| \le \frac{1}{\sqrt{\Delta}} \left \| \mathbb{K}_h \right \|_\infty \Delta \le {\modch \frac{\left \| K \right \|_\infty}{h (T)} \sqrt{2}}.$$
{\modch We choose $h(T):= (\frac{1}{T})^{\frac{1}{2} - \epsilon} = (\frac{1}{n \Delta})^{\frac{1}{2} - \epsilon} \ge c (\frac{1}{n})^{(\frac{1}{2} - \epsilon)}$, for some $\epsilon \in (0, \frac12)$.}
Hence, assumption (i) holds true {\modch with $M_n := c n^{\frac{1}{2} - \epsilon}$}. Concerning assumption (ii) we remark that, for any $i \ge 1$ and any $1 \le r \le m$, 
\begin{equation*} \begin{split}
\mathbb{E}[(Y_{n,i}^{(r)})^2] &= \text{Var}\left(\frac{1}{\sqrt{\Delta}} \int_0^{\Delta} \mathbb{K}_h (x_r - X_u) du\right) = \text{Var}(\sqrt{\Delta} \hat{\mu}_{h, \Delta} (x_r) )\\
&= \Delta \text{Var}(\hat{\mu}_{h, \Delta} (x_r) ) \le \Delta \frac{c}{\Delta} = c,
\end{split}
\end{equation*}
where in the last inequality we have used (\ref{eq: upper bound variance}). We next check condition (iii). Let $b_n$ be a sequence of integers such that $b_n \rightarrow \infty$ and $b_n \le n$ for every $n$. For every $1 \le r \le m$ and $1 \le s \le m$, we have  
\begin{equation*} \label{eq: conv 1.4}\begin{split}
&\frac{1}{b_n} \mathbb{E}\left[\sum_{i = 1}^{b_n} Y_{n,i}^{(r)} \sum_{j = 1}^{b_n} Y_{n,j}^{(s)} \right] = \frac{1}{{\modch\Delta} b_n} \int_0^{{\modch\Delta} b_n} \int_0^{{\modch\Delta} b_n} \text{Cov}(\mathbb{K}_h (x_r - X_u), \mathbb{K}_h (x_s - X_v)) du \, dv \\
&= 2 \int_0^{{\modch\Delta} b_n} (1 - \frac{u}{{\modch\Delta} b_n}) \int_{\mathbb{R}} \int_{\mathbb{R}} \mathbb{K}_h (x_r - z_1) \mathbb{K}_h (x_s - z_2) g_u (z_1, z_2) dz_1 \, dz_2 \, du\\
&= 2 \int_{\mathbb{R}} \int_{\mathbb{R}} \int_0^{ {\modch\Delta} b_n} (1 - \frac{u}{{\modch\Delta} b_n}) K(w_1) K(w_2) g_u (x_r - h(T) w_1, x_s - h(T) w_2) du \, dw_1 \, dw_2,
\end{split}
\end{equation*}
where we have used Fubini's theorem and the change of variables $w_1:= \frac{x_r - z_1}{h(T)} $, $w_2:= \frac{x_s - z_2}{h(T)} $.
Using dominated convergence and the fact that $h(T) \rightarrow 0$ for $T \rightarrow \infty$ {\modch and $\Delta b_n \rightarrow \infty$ for $n \rightarrow \infty$ as $\Delta \ge 1$}, we obtain
\begin{equation*} \begin{split}
\lim_{n \rightarrow \infty} \frac{1}{b_n} \mathbb{E}\left[\sum_{i = 1}^{b_n} Y_{n,i}^{(r)} \sum_{j = 1}^{b_n} Y_{n,j}^{(s)} \right] &= 2 \int_\mathbb{R} K(w_1) \int_\mathbb{R} K(w_2) \int_0^\infty g_u(x_r, x_s) du \, dw_2 \, dw_1 \\
&= 2 \int_0^\infty g_u(x_r, x_s) du =: \sigma (x_r, x_s), 
\end{split}
 \end{equation*}
which proves (iii). Remark that it is possible to use dominated convergence theorem since we have shown in the proof of Proposition \ref{prop: WCL satisfied} that 
$$
\sup_{y \in \mathbb{R}} \vert g_u(x,y) \vert \leq c \left( u^{-1/2} {\bf 1}_{\{u \leq 2\}}+\mu(x) (1+f^{\ast}(x))e^{-\rho u} {\bf 1}_{\{u > 2\}}\right),$$
for some positive constants $c$ and $\rho$.
In particular, we have 
\begin{align*}
& |(1 - \frac{u}{ \Delta b_n}) K(w_1) K(w_2) g_u (x_r - h(T) w_1, x_s - h(T) w_2) 1_{[0, b_n]} (u) 1_{\mathbb{R}^2}(w_1, w_2)|  \\
& \le c \left( u^{-1/2} {\bf 1}_{\{u \leq 2\}}+ e^{\epsilon (\vert x_r \vert +\vert w_1 \vert)}   e^{-\rho u} {\bf 1}_{\{u > 2\}}\right) |K(w_1) K(w_2)| \in L^1(\mathbb{R}^+ \times \mathbb{R}^2), 
\end{align*}
as $K$ has support on $[-1,1]$.

We now check (iv). We remark that if a process is $\beta$-mixing, then it is also $\alpha$-mixing and the following estimation holds (see Theorem 3 in Section 1.2.2 of \cite{Mixing})
$$\alpha_k \le \beta_{Y_{n,i}} (k) = \beta_X (k) \le c e^{- \gamma_1 k}.$$
Therefore, it suffices to show that there exists ${\modch \gamma_0} \in (1, \infty)$ such that
$$\sum_{k \ge 1} k e^{- k \gamma_1 \frac{({\modch \gamma_0}-1)}{{\modch \gamma_0}}} < \infty,$$
which is true for any ${\modch \gamma_0} > 1$, so (iv) is satisfied.

We are left to show (v). Set $f({\modch \gamma_0}) := \frac{{\modch \gamma_0}^2}{(3 {\modch \gamma_0} - 1) (2 {\modch \gamma_0} -1)}$ and observe that  $f(1)=\frac{1}{2}$ and for $\gamma_0>1$, $f$ is continuous, strictly decreasing, and $\frac{1}{6}<f(\gamma_0) <\frac{1}{2}$.
Therefore, given $\epsilon \in (0, \frac 12)$, there always exists $\gamma_0>1$ such that for all $n \geq 1$,
\begin{equation*} 
 n^{\frac{1}{2} - \epsilon} \le n^{f( \gamma_0)}.
\end{equation*}

Thus, condition (v) is satisfied.
We can then apply Theorem \ref{th: teorema 1.1} which directly leads us to \eqref{eq: as norm 1} and concludes the desired proof.

\subsection{Proof of Theorem \ref{th: lower bound}}

 The proof of of Theorem \ref{th: lower bound} follows as the proof of the lower bound for $d \geq 3$ obtained in Theorem 3 of \cite{lowerbound}. Therefore, we will only explain the main steps and the principal differences. 

 \vskip 5pt
 {\it Step 1} The first step consists in showing that given a density function $f$, we can always find a drift function $b_f$ such that $f$ is the unique invariant density function of equation (\ref{eq: model lower bound}) with drift coefficient $b=b_f$. We  give the statement and proof in dimension $d=1$, as in  Propositions 2 and 3 of \cite{lowerbound} it is only done for $d \geq 2$.
 \begin{proposition} \label{prop: b pi proprieta}
 Let $f : \mathbb{R} \rightarrow \mathbb{R}$ be a $\mathcal{C}^2$ positive probability density satisfying the following conditions 
\begin{enumerate}
    \item $\lim_{y \rightarrow \pm \infty} f (y) = 0$ and $\lim_{y \rightarrow \pm \infty} f' (y) = 0$.
    \item There exist {\modch $\hat{c}_1 > 0$ and} $0<\epsilon < \frac{\epsilon_0}{|\gamma|} $, where $\epsilon_0$ is as in Assumption {\bf A4} such that, for any $y, z \in \mathbb{R}$, 
    $$f(y \pm z) \le \hat{c}_1 e^{\epsilon |z| }f (y).  $$
    \item For $\epsilon > 0$ as in 2. there exists $\hat{c}_2(\epsilon) > 0$ such that 
   $$\sup_{y < 0} \frac{1}{f(y)} \int_{- \infty}^y f (w) dw < \hat{c}_2 \; \text{ and } \; \sup_{y >0} \frac{1}{f(y)} \int_{y}^\infty f (w) dw < \hat{c}_2 .$$
    \item There exists $0 <\tilde{\epsilon} < \frac{a^2}{2 \gamma^2 c_4  \hat{c}_2 \hat{c}_4 \hat{c}_1}$ and $R>0$ such that for any $\vert y \vert>R$, $\frac{f'(y)}{f(y)} \le - \tilde{\epsilon} \sgn(y)$, where $c_4$ is as in Assumption {\bf A4}. Moreover, there exists $\hat{c}_3$ such that
   for any $y \in \mathbb{R}$, $\vert f'(y) \vert \leq \hat{c}_3 f(y)$.
    \item For any $y \in \mathbb{R}$ and $\tilde{\epsilon}$ as in 4.
    $|f''(y)| \le \hat{c}_4\tilde{\epsilon}^2 f (y).$
\end{enumerate}
Then there exists a bounded Lipschitz function $b_f$ which satisfies {\bf A2} such that $f$ is the unique invariant density to equation \textnormal{(\ref{eq: model lower bound})} with drift coefficient $b=b_f$.
\end{proposition}

\begin{proof}
Let $A_d$ be the discrete part of the generator of the diffusion process $X$ solution of \eqref{eq: model lower bound} and let $A^*_d$ its adjoint.
We define $b_f$ as 
\begin{equation*}
b_f (x) = \begin{cases}
\frac{1}{f (x)} \int_{- \infty}^{x} (\frac{1}{2} a^2 f''(w) + A^*_d \, f (w)) d w, \quad &\mbox{if } x < 0;\\
 -\frac{1}{f (x)} \int_{x}^{\infty} \frac{1}{2} a^2 f''(x) (w) + A^*_d \, f (w) d w, \quad &\mbox{if } x > 0,
 \end{cases}
 \end{equation*}
 where
 $$A^*_d \, f (x) = \int_{\mathbb{R} }[f(x - \gamma z) - f(x) + \gamma z f'(x)] F(z) dz.$$
 Then, following Proposition 3 in \cite{lowerbound}, one can check  that $b_f$ is bounded, Lipschitz, and satisfies {\bf  A2}.
 Moreover, if we replace $b$ by $b_f$ in equation (\ref{eq: model lower bound}), then $f$ is the unique invariant density.
\end{proof}

{\it Step 2} The second step consists in defining two probability density functions $f_0$ and $f_1$ in
$\mathcal{H}_1(\beta, \mathcal{L})$. 

We first define $f_0(y)=c_{\eta} f(\eta \vert y \vert)$, where $\eta \in (0, \frac{1}{2})$, $c_{\eta}$ is such that $\int f_0=1$, where $f$ is defined as follows. We first consider the piecewise function
\begin{equation*}
g(x)=\begin{cases}
e^{-\vert x \vert}, &\text{ if } \vert x \vert \geq 1 \\
e^{-4(\vert x \vert-\frac12)^2}, &\text{ if } \frac12<\vert x \vert < 1 \\
1, &\text{ if } \vert x \vert \leq \frac12.
\end{cases}
\end{equation*}
Observe that $g$ is continuous, satisfies $\frac12 e^{-\vert x \vert} \leq g(x) \leq 2 e^{-\vert x \vert}$ for all $x \in \mathbb{R}$, and each piece belongs to $C^{\infty}$ and has bounded derivatives. We define $f$ as a $\mathcal{C}^{\infty}$ approximation of $g$, with bounded derivatives of all orders and satisfying
\begin{equation} \label{eq: ass f}
\frac12 e^{-\vert x \vert} \leq f(x) \leq 2 e^{-\vert x \vert}, \quad
\vert f'(\vert x \vert) \vert \leq 5 e^{-\vert x \vert}, \quad \text{and}\quad 
\vert f''(\vert x \vert) \vert \leq 14 e^{-\vert x \vert}.
\end{equation}
Observe that the two latter inequalities are satisfied by $g$ piecewise. 

It is easy to see that $\eta$ can be chosen small enough so that $f_0 \in \mathcal{H}_1(\beta, \mathcal{L})$.
Indeed, first, it is clear that all the derivatives of $f_0$ can be bounded by the constant $\mathcal{L}$ for $\eta$ small enough. Furthermore, the following bounds hold true for any $x$ and $t$ in $\R$ 
\begin{align*}
& |D^{\lfloor \beta \rfloor} f_0 (x + t) - D^{\lfloor \beta \rfloor} f_0 (x) | \\
& \le |D^{\lfloor \beta \rfloor} f_0 (x + t) - D^{\lfloor \beta \rfloor} f_0 (x) |^{\beta  -\lfloor \beta \rfloor} (2 \left \| D^{\lfloor \beta \rfloor} f_0 \right \|_\infty)^{1 - (\beta - \lfloor \beta \rfloor)} \\
& \le \left \| D^{\lfloor \beta \rfloor + 1} f_0 \right \|_\infty^{\beta - \lfloor \beta \rfloor} (2 \left \| D^{\lfloor \beta \rfloor} f_0 \right \|_\infty)^{1 - (\beta - \lfloor \beta \rfloor)} \, |t|^{\beta - \lfloor \beta \rfloor}.
\end{align*}
Again, it suffices to choose $\eta$ small enough to ensure that $$\left \| D^{\lfloor \beta \rfloor + 1} f_0 \right \|_\infty^{\beta - \lfloor \beta \rfloor} (2 \left \| D^{\lfloor \beta \rfloor} f_0 \right \|_\infty)^{1 - (\beta - \lfloor \beta \rfloor)} \le \mathcal{L},$$
which shows that $f_0 \in \mathcal{H}_1(\beta, \mathcal{L}) \subset \mathcal{H}_1(\beta, 2 \mathcal{L})$. 

We also ask that the constant $c_4$ in Assumption {\bf A4} is such that
\begin{equation}
c_4<\frac{a^2}{2\gamma^2 4^{2} 28}.
\label{eq: cond c4}
\end{equation}
This means that the jumps have to integrate an exponential function. The bound depends on the coefficients $a$ and $\gamma$ and so it depends only on the model. 

Under the conditions above it is easy to see that $f_0$ satisfies the assumptions of Proposition \ref{prop: b pi proprieta} with $\hat{c}_1=4$, $\epsilon=\eta$, $\hat{c}_2=\frac{4}{\eta}$, $R=\frac{1}{\eta}$, $\tilde{\epsilon}=\eta$ $\hat{c}_3=28 \eta$, and
$\hat{c}_4=28 $. {\modch Indeed, point 1 of Proposition \ref{prop: b pi proprieta} clearly holds true from the definition of $f_0$. To show the second point we observe that, thanks to \eqref{eq: ass f}, we have 
$$
f_0 (y \pm z) = c_n f(\eta |y \pm z|)  \le 2 c_n e^{- \eta |y|} e^{\eta |z|}  \le  4 f_0(y) e^{\eta |z|},
$$
which implies point 2 with $\hat{c}_1 = 4$ and $\epsilon = \eta$, since we can choose $\eta$ small enough to make the condition on $\epsilon$ satisfied. 
In order to prove point 3 we use again \eqref{eq: ass f}. It follows that, for any $y < 0$, 
\begin{align*}
   \frac{1}{f_0 (y)} \int_{- \infty}^y f_0(w) dw & = \frac{1}{c_n f(\eta |y|)} \int_{- \infty}^y c_n f(\eta |w|) dw \\
   & \le 2 e^{\eta |y|} \int_{- \infty}^y 2 e^{- \eta w} dw 
   = 4 e^{\eta |y|} \frac{e^{- \eta |y|}}{\eta} = \frac{4}{\eta}.
\end{align*}
For $y > 0$ an analogous reasoning applies. Thus, $f_0$ satisfies the third point with $\hat{c}_2(\epsilon) = \hat{c}_2(\eta) = \frac{4}{\eta}$. 
For the fourth point, we observe that, for $|y| > \frac{1}{\eta}$, 
$$f_0(y) = - \eta \, \text{sgn}(y) f_0(y).$$
That is, the first part of point 4 holds true for $|y| > R$, taking $R= \frac{1}{\eta}$ and $\tilde{\epsilon} = \eta$. 
Moreover, we observe that using \eqref{eq: ass f} we have, for $k = 1, 2$, 
$$
|f_0^{(k)}(y)| = |c_n f^{(k)}(\eta |y|)| \le 14 c_n \eta^k e^{- \eta |y|} \le  28 \eta^k f_0(y).
$$
This shows that both the fourth and the fifth points  hold true, with $\hat{c}_3(\eta) =  28 \eta$ and $\hat{c}_4 =  28$. 
Finally, we need to check that the condition on $\tilde{\epsilon}$ given in the fourth point which writes as
$$\tilde{\epsilon} = \eta < \frac{a^2}{2 \gamma^2 c_4 \hat{c}_2 \hat{c}_4 \hat{c}_1} = \frac{a^2 \, \eta}{2 \gamma^2 c_4\,  4 \, 28\, 4 },$$
which is equivalent to \eqref{eq: cond c4}. Hence, $f_0$ satisfies all the assumptions in Proposition \ref{prop: b pi proprieta}.} 

Therefore, $b_{0}:=b_{f_0}$ belongs to  $\Sigma(\beta, \mathcal{L})$. Recall that $b_{0}$ belongs to  $\Sigma(\beta, \mathcal{L})$ if and only if $f_0$ belongs to $\mathcal{H}_1(\beta, {\modch 2} \mathcal{L})$ and $b_{0}$ is bounded, Lipschitz and satisfies the drift condition \textbf{A2}.

We next define 
\begin{equation} \label{f1}
f_1(x)=f_0(x)+\frac{1}{M_T} \hat{K}\left(\frac{x-x_0}{{\modch H}}\right),
\end{equation}
where $x_0 \in \R$ is fixed and $\hat{K}: \mathbb{R}\rightarrow \mathbb{R}$ is a $C^{\infty}$
function with support on $[-1, 1]$ such that
$$
\hat{K}(0)=1, \quad \int_{-1}^1 \hat{K}(z) dz=0.
$$
{\modch Here $H$ is a constant and $M_T$ will be calibrated later and satisfies that $M_T \rightarrow \infty$ as $T \rightarrow \infty$. Observe that in the proof
of the lower bound for the case $d \geq 3$ presented in \cite{lowerbound}, $H$ is a function of $T$ converging to 0 as $T \rightarrow \infty$. For the case $d=1$, it suffices
to chose it 
constant and we will see below that the same computations done in \cite{lowerbound} will work in this case and it suffices to calibrate $M_T$.} 

Then it can be shown as in \cite[Lemma 3]{lowerbound} that 
if {\modch for all $\epsilon>0$ and $T$ sufficiently large, 
\begin{equation}
\frac{1}{M_T}\leq \epsilon H^{\beta} \qquad \mbox{and} \qquad \frac{1}{H}=o(M_T)
\label{eq: cond MT}
\end{equation}
as $T \rightarrow \infty$, then if $\epsilon>0$ is small enough we have that
 $b_{1}:=b_{f_1}$ belongs to  $\Sigma(\beta, \mathcal{L})$ for $T$ sufficiently large. Indeed, on one hand, (\ref{eq: cond MT}) is clearly true when $H$ is a constant. On the other hand, the same argument used in \cite[Lemma 3]{lowerbound} applies to show that $f_1$ belongs to $\mathcal{H}_{1}(\beta, 2 \mathcal{L})$ when $H$ is a constant, up to choose $\epsilon$ in \eqref{eq: cond MT} smaller than a constant depending on $\mathcal{L}$ and $H$. }

\vskip 5pt
{\it Step 3} As $ b_0, b_1  \in \Sigma (\beta, \mathcal{L})$, we can  write
$$R(\tilde{\mu}_T (x_0)) \ge \frac{1}{2} \mathbb{E}_{1}^{(T)}[(\tilde{\mu}_T (x_0) - f_1 (x_0))^2] + \frac{1}{2} \mathbb{E}_{0}^{(T)}[(\tilde{\mu}_T (x_0) - f_0 (x_0))^2],$$
where $\mathbb{E}_{i}^{(T)}$ denotes the  expectation with respect to $b_i$.
Then, following as in \cite{lowerbound}, using Girsanov's formula, we  can show that if 
\begin{equation}\label{eq:final cond MT}
\sup_{T \ge 0} T \frac{1}{M_T^2 {\modch H}} < \infty,
\end{equation}
then for sufficiently large $T$,
\begin{equation}\label{eq: risk con Mt dim 1}
R(\tilde{\mu}_T (x_0)) \ge \frac{C}{8 \lambda} \frac{1}{M_T^2},
\end{equation}
where the constants $C$ and $\lambda$ are as in Lemma 4 of \cite{lowerbound} and they do not depend on the point $x_0$.  We finally look for the larger choice of $\frac{1}{M^2_T}$ for which both \eqref{eq: cond MT} and \eqref{eq:final cond MT} hold true. It suffices to choose $M_T=\sqrt{T}$ to conclude the proof of Theorem \ref{th: lower bound}.
 
\begin{remark} \label{d12}
The two hypothesis method used above 
does not work to prove the 2-dimensional lower bound of Theorem \ref{th: lower d=2}. 
Indeed, following as above, we can define 
$$f_1(x)=f_0(x)+\frac{1}{M_T} \hat{K}\left(\frac{x-x_0}{H_1(T)}\right)\hat{K}\left(\frac{x-x_0}{H_2(T)}\right).$$
Then, it is possible to show that \eqref{eq: risk con Mt dim 1} still holds and, therefore, we should take $M_T$ such that $\frac{1}{M^2_T}= \frac{\log T}{T}$. On the other hand, condition \eqref{eq:final cond MT} now becomes
$$\sup_{T \ge 0} T \frac{1}{M^2_T} \left(\frac{H_2(T)}{H_1(T)} + \frac{H_1(T)}{H_2(T)}\right) < \infty.$$
The optimal choice of the bandwidth is achieved for $H_2(T) = H_1(T)$ which yields to
$\sup_{T \ge 0} T \frac{1}{M^2_T} < \infty,$
which is clearly not satisfied when  $\frac{1}{M^2_T}= \frac{\log T}{T}$.  
\end{remark}

\subsection{Proof of Theorem \ref{th: lower d=2}}

We will apply Lemma \ref{lemma: lower dim 2} with $\psi:= v\sqrt{\frac{\log T}{T}}$, where $v>0$ is fixed. As above we divide the proof into three steps.

\vskip 5pt

 {\it Step 1} As in  the one-dimensional case, the first step consists in showing that given a density function $f$, we can always find a drift function $b_f$ such that $f$ is the unique invariant density function of equation (\ref{eq: model lower bound}) with drift coefficient $b=b_f$, which is proved in  Propositions 2 and 3 of \cite{lowerbound}. We remark that condition (\ref{conditiona}) is needed in Proposition 3 to ensure that the terms on the diagonal of the volatility coefficient $a$ dominate on the others, which is crucial to get that $b_f$ satisfies the drift condition \textbf{A2}. 
 \vskip 5pt
{\it Step 2} We next define the probability density $f_0 \in \mathcal{H}_2(\beta, \mathcal{L}) $, the finite set $J_T$, and the set of probability densities $\left \{ f_j, \, j \in J_T \right \} \subset \mathcal{H}_2(\beta, \mathcal{L})$ needed in order to apply Lemma \ref{lemma: lower dim 2}.

We first define $f_0$ as $\pi_0$ in Section 7.2 of \cite{lowerbound}, which is the two-dimensional version of $f_0$ defined in the proof of Theorem \ref{th: lower bound}, that is,
\begin{equation} \label{f0}
f_0(x)= c_\eta f(\eta (aa^T)^{-1}_{11} |x_1|) f(\eta (aa^T)^{-1}_{22} |x_2|),\quad x = (x_1, x_2) \in \mathbb{R}^2,
\end{equation}
where $f$ is as in Step 2 of the proof of Proposition \ref{prop: b pi proprieta}.
The density $f_0$ belongs to $\mathcal{H}_2(\beta, \mathcal{L})$ by construction. 

We then set 
\begin{equation} \label{eq: def JT}
  J_T := \left \{ 1,\ldots, \lfloor \frac{1}{\sqrt{H_1} } \rfloor \right \} \times \left \{ 1,\ldots, \lfloor \frac{1}{\sqrt{H_2} } \rfloor \right \}, 
\end{equation}
{\modch  where in order to lighten the notation we will write $H_1$ and $H_2$ for $H_1(T)$ and $H_2(T)$, respectively}, which are two quantities that converge to $0$ as $T \rightarrow \infty$ and need to be calibrated. 

Finally, for $j:= (j_1, j_2) \in J_T$, we define $x_j:= (x_{j,1}, x_{j,2}) = (2 j_1 H_1, 2 j_2 H_2 )$ and we set
$$f_j(x) := f_0(x) + {\modch 2}v \sqrt{\frac{\log T}{T}} \hat{K}\left(\frac{x_1 - x_{j,1}}{H_1}\right) \hat{K}\left(\frac{x_2 - x_{j,2}}{H_2}\right),$$
where recall that $v>0$ is fixed and $\hat{K}$ is as in (\ref{f1}).

Acting as in Lemma 3 of \cite{lowerbound}, recalling that the rate $\frac{1}{M_T}$ therein is now replaced by $\sqrt{\frac{\log T}{T}}$ (see also points 1. and 3. in the proof of Proposition \ref{prop: constante per d=2} below), it is easy to see that if there exists $\epsilon > 0$ sufficiently small such that for large $T$, 
\begin{equation} \label{eq: cond holder dim 2}
\sqrt{\frac{\log T}{T}} \le \epsilon H_1^{\beta_1}, \qquad \sqrt{\frac{\log T}{T}} \le \epsilon H_2^{\beta_2},
\end{equation}
then, for any $j \in J_T$ and large $T$,
$b_j \in \Sigma(\beta, \mathcal{L}).$
In particular, 
$f_j \in \mathcal{H}_2(\beta, \mathcal{L}).$
Therefore, $\left \{ f_j, \, j \in J_T \right \} \subset \mathcal{H}_2(\beta, \mathcal{L})$.

{\modch 
In order to evaluate the difference between $f_j$ and $f_k$ we remark first of all that, as $\hat{K}$ has support on $[-1, 1]$, $\prod_{l = 1}^2 \hat{K}(\frac{x_l -x_{j,l}}{H_l})$ is different from $0$ only if $|\frac{x_l -x_{j,l}}{H_l}| \le 1$ for any $l \in \{ 1, 2 \}$. Then, 
\begin{align*}
\left \| f_j - f_k \right \|_\infty  & \ge |f_j (x_j) - f_k (x_j)| \\
& = 2 v \sqrt{\frac{\log T}{T}} [\prod_{l = 1}^2 \hat{K}(\frac{x_{j,l} -x_{j,l}}{H_l}) - \prod_{l = 1}^2 \hat{K}(\frac{x_{j,l} -x_{k,l}}{H_l})] \\
& =  2 v \sqrt{\frac{\log T}{T}} \prod_{l = 1}^2 \hat{K}(0) = 2 v\sqrt{\frac{\log T}{T}} = 2 \psi,
\end{align*}
where we have used that, as $j \neq k$, there is a $l_0 \in \{ 1, 2 \}$ such that $l_0 \neq k_0$ and so in particular, by construction, $|j_{l_0} - k_{l_0}| \ge 1$. It follows that
$$|\frac{x_{j,l_0} -x_{k,l_0}}{H_{l_0}}| = |\frac{2 j_{l_0}H_{l_0} - 2 k_{l_0}H_{l_0}}{h_{l_0}}| \ge 2 $$
and so the kernel evaluated in this point is null.
This proves the first  condition of Lemma \ref{lemma: lower dim 2}.}

\vskip 5pt
{\it Step 3} We are left to show the remaining conditions of Lemma \ref{lemma: lower dim 2}.
The absolute continuity $\mathbb{P}^{(T)}_j \ll \mathbb{P}^{(T)}_0$ and the expression for 
$\frac{d \mathbb{P}^{(T)}_j}{d \mathbb{P}^{(T)}_0}(X^T)$ are both obtained by Girsanov formula, as in Lemma 4 of \cite{lowerbound}. We have, $$KL (\mathbb{P}^{(T)}_j, \mathbb{P}^{(T)}_0) = \mathbb{E}^{(T)}_j \left[\log \left(\frac{f_j}{f_0}(X^T)\right)\right] + \frac{1}{2} \mathbb{E}^{(T)}_j \left[\int_0^T  \vert a^{-1} (b_{0} (X_u)- b_{j}(X_u))  \vert^2 du\right],$$
where the law of $X^T=(X_t)_{t \in [0,T]}$ under $\mathbb{P}^{(T)}_j$ is the one of the solution to equation (\ref{eq: model lower bound}) with $b=b_{0}$.

By the definition of the $f_j$'s it is easy to see that the first term is $o(1)$ as $T \rightarrow \infty$. In fact, as $\hat{K}$ is supported in $[-1,1]$, 
\begin{equation*} \begin{split}
\mathbb{E}^{(T)}_j \left[\log \left(\frac{f_j}{f_0}(X^T)\right)\right] 
&=\int_{\R^2} \log\bigg(1+ \frac{ {\modch 2} v \sqrt{\frac{\log T}{T}} \hat{K}\left(\frac{x_1 - x_{j,1}}{H_1}\right) \hat{K}\left(\frac{x_2 - x_{j,2}}{H_2}\right)}{f_0(x)}\bigg) f_0(x) dx \\
& \leq \bigg\vert\log\bigg(1+c_{\ast} v \sqrt{\frac{\log T}{T}}  \| \hat{K} \|^2_\infty\bigg)\bigg\vert,
\end{split}
\end{equation*}
which tends to zero as $T  \rightarrow \infty$, where   $c_{\ast}:=\frac{{\modch 8}}{c_\eta} e^{4 \eta \, k} $, $c_\eta$ is the constant of normalization introduced in the definition of $f_0$, and $k:= \max_{i=1, 2 } (a a^T)^{-1}_{i i}$. In fact, this follows from the definition of $f_0$ in (\ref{f0}).
 Since $f(x) \ge \frac{1}{2} e^{- |x|}$, we obtain
$$\frac{1}{f_0 (x)} \le \frac{1}{c_\eta} \frac{2}{e^{- \eta (a a^T)^{-1}_{11} |x_1| }} \frac{2}{e^{- \eta (a  a^T)^{-1}_{22} |x_2| }} \le \frac{4}{c_\eta}e^{\eta k (|H_1| + |x_{j,1}| + |H_2| + |x_{j,2}| )},$$
where we have also used the fact that, as $\hat{K}$ is supported in $[-1, 1]$, we have $x \in [x_{j,1} - H_1, x_{j,1} + H_1] \times [x_{j,2} - H_2, x_{j,2} + H_2]$. Finally, by the definition of $x_j$ and the fact that $H_i \rightarrow 0$ as $T \rightarrow \infty$ for $i = 1,2$ (and so for $T$ large enough they are smaller than 1), we get
\begin{equation}
\frac{1}{f_0(x)} \le \frac{4}{c_\eta}e^{4 \eta k} \quad \mbox{for any } x \in [x_{j,1} - H_1, x_{j,1} + H_1] \times [x_{j,2} - H_2, x_{j,2} + H_2].
\label{eq: stima f0}
\end{equation}

Regarding the second term, using the stationarity of the process $X^T$, we have 
$$
\mathbb{E}^{(T)}_j \left[\int_0^T  \vert a^{-1} (b_{0} (X_u)- b_{j}(X_u))  \vert^2 du\right]
=T\int_{\R^2} \vert a^{-1}(b_{0} (x)- b_{j}(x)) \vert^2 f_0(x) dx .
$$

Then, the  following asymptotic bound will be  proved at  the end of this Section.
\begin{proposition} \label{prop: constante per d=2}
For $T$ large enough, 
\begin{align*}
\int_{\R^2} \vert a^{-1}(b_{0} (x)- b_{j}(x)) \vert^2 f_0(x) dx \le  64 \frac{e^{8 \eta k}}{c_\eta^2} k^2 v^2 H_1 H_2 \left(\frac{1}{H_1} + \frac{1}{H_2}\right)^2 \frac{\log T}{T}.
\end{align*}
\end{proposition}

Taking the optimal choice for the bandwidth in Proposition \ref{prop: constante per d=2}, which is $H_1 = H_2$, we get that
$$\int_{\R^2} \vert a^{-1}(b_0 (x)- b_j(x)) \vert^2 f_0(x) dx \le 64 \frac{e^{8 \eta k}}{c_\eta^2} k^2 v^2 4 \frac{\log T}{T}. $$
In particular, after having ordered $\beta_1 \le \beta_2$, we choose $H_1= H_2= (\frac{\log T}{T})^{{\modch \alpha}}$ with ${\modch \alpha} \le \frac{1}{2 \beta_2}= (\frac{1}{2 \beta_1} \land \frac{1}{2 \beta_2})$ so that condition \eqref{eq: cond holder dim 2} is satisfied. We therefore get
\begin{equation*} \begin{split}
KL (\mathbb{P}^{(T)}_{j}, \mathbb{P}^{(T)}_{0}) \le 128 \frac{e^{8 \eta k}}{c_\eta^2} k^2 \, v^2  \log T   \le 128 \frac{e^{8 \eta k}}{c_\eta^2 {\modch \alpha}} k^2 \, v^2 \log (|J_T|),
\end{split}
\end{equation*}
being the last estimation a consequence of the fact that, by construction, 
$$\log(|J_T|)\ge {\modch \alpha} \log \left(\frac{T}{\log T}\right)={\modch \alpha} \log(T)(1+o(1)).$$
It is therefore enough to choose $v$ such that $128 \frac{e^{8 \eta k}}{c_\eta^2 {\modch \alpha}} k^2 \, v^2 < \frac{1}{8}$ (ie $v^2 < \frac{c_\eta^2 {\modch \alpha}}{1024 \, k^2 e^{8 \eta k}}$) and apply Lemma \ref{lemma: lower dim 2} to  conclude the proof of Theorem \ref{th: lower d=2}.

\subsection{Proof of Proposition \ref{prop: constante per d=2}}

The proof of Proposition \ref{prop: constante per d=2} follows similarly as Proposition 4 of \cite{lowerbound}. Indeed, we first define the set
$$K_T^j:= [x_{j,1} - H_1, x_{j,1} + H_1 ] \times [x_{j,2} - H_2, x_{j,2} + H_2 ],$$
{\modch where we recall that we write $H_1$ and $H_2$ for $H_1(T)$ and $H_2(T)$, respectively, in order to simplify the notation.} Then we show the following points for $T$ large enough:
\begin{enumerate}
\item {\modch There exists a constant $c > 0$ such that, for any $x$ in the complementary set of $K_T$, that we denote as $K_T^{j\,c}$, and for any $i \in \left \{ 1,2 \right \}$,} $$|b^i_{j}(x) - b^i_{0} (x)| \le c \, v\, \sqrt{\frac{\log T}{T}}.$$
\item {\modch There exists a constant $c > 0$ such that, for any $ i \in \left \{ 1, 2 \right \}$,} 
$$\int_{K_T^{j \, c}} |b^i_{j}(x) - b^i_{0} (x)| f_0 (x) dx \le c \, v\, \sqrt{\frac{\log T}{T}} H_1 H_2.$$
\item For any $x \in K_T^j$ and $i \in \left \{ 1, 2 \right \},$ $$|b^i_{j}(x) - b^i_{0} (x)| \le \frac{8 }{c_{\eta}} e^{4 \eta k} k v \sqrt{\frac{\log T}{T}} \left(\frac{1}{H_1} + \frac{1}{H_2}\right).$$
\end{enumerate}

The proof of the first two points follows exactly the one in Proposition 4 of \cite{lowerbound}, remarking that 
$$d_T (x) := \pi_1 (x) - \pi_0 (x) = \frac{1}{M_T} \prod_{l = 1}^d K\left(\frac{x_l - x_0^l}{h_l(T)}\right)$$
in \cite{lowerbound} is now replaced by
$$d_T^j (x) := f_j (x) - f_0 (x) = {\modch 2} v \, \sqrt{\frac{\log T}{T}} \hat{K}\left(\frac{x_1 - x_{j,1}}{H_1}\right) \hat{K}\left(\frac{x_2 - x_{j,2}}{H_2}\right),$$
and the set $$K_T := [x_0^1 - h_1 (T), x_0^1 + h_1(T) ] \times \cdots \times [x_0^d - h_d (T), x_0^d + h_d (T)]$$
introduced in \cite{lowerbound} is now replaced by $K_T^j$. We recall that $K$ and $\hat{K}$ are exactly the same kernel function. The proof of Proposition 4 of \cite{lowerbound} is based on the fact that $d_T(x)$ and its derivatives are null for $x \in K_T^c$. In the same way, $d_T^j(x)$ and its derivatives are null for $x \in K_T^{j\, c}$. Then, acting as in \cite{lowerbound}, it is easy to see that the first two points above hold true. 

Comparing the third point above with the third point of Proposition 4 of \cite{lowerbound}, it is clear that our goal is  {\modch to show that the constant $c$ that appears in the third point of Proposition 4 of \cite{lowerbound} 
is explicit and equal to $\frac{8 }{c_{\eta}} e^{4 \eta k} k$ when $d=2$}. Keeping the notation in \cite{lowerbound}, we first introduce the following quantities:
$$\tilde{I}^i_1[f_0 ] (x) := \frac{1}{2} \sum_{j = 1}^2 (a a^T)_{i j}\frac{\partial f_0}{ \partial x_j} (x),\qquad \tilde{I}^i_2[f_0 ] (x) =  \int_{- \infty}^{x_i} A^*_{d, i} f_0(w_i) dw. $$
We moreover introduce the notation
$$\tilde{I}^i [f_0]  (x) = \tilde{I}^i_1[f_0 ] (x) + \tilde{I}^i_2[f_0 ] (x).$$
According with the definition of $b$, we have 
$$b^i_{0} (x) = \frac{1}{f_0 (x)}\tilde{I}^i[f_0 ] (x), \qquad b^i_{j} (x) = \frac{1}{f_j (x)}\tilde{I}^i[f_j ] (x).$$
Since the operator $f \rightarrow \tilde{I}^i [f]$ is linear, we deduce that
\begin{equation} \label{eq: b pi1 primo paragone}
b^i_{j} (x) = \frac{1}{f_j (x)}\tilde{I}^i[f_j]  (x) = \frac{1}{f_j (x)}\tilde{I}^i[f_0 ] (x) + \frac{1}{f_j (x)}\tilde{I}^i[d_T^j ] (x).
\end{equation}
Therefore, 
$$b^i_{j} - b^i_{0} = (\frac{1}{f_j} - \frac{1}{f_0}) \tilde{I}^i[f_0] + \frac{1}{f_j} \tilde{I}^i[d_T^j] = \frac{f_0 - f_j}{f_j} \frac{1}{f_0} \tilde{I}^i[f_0] + \frac{1}{f_j} \tilde{I}^i[d_T^j] = \frac{d_T^j}{f_j} b^i_{0} + \frac{1}{f_j} \tilde{I}^i[d_T^j]. $$
We need to evaluate such a difference on the compact set $K_T^j$. For this, we will  use that  fact that $f_j = f_0 + d_T^j$, and obtain  a lower bound away from $0$. Specifically, from the definition of $d_T^j$, we get
\begin{equation} \label{eq: norma inf dT}
\left \| d_T^j \right \|_\infty \le {\modch 2} v \sqrt{\frac{\log T}{T}}  \| \hat{K}  \|_\infty^2 = {\modch 2} v \sqrt{\frac{\log T}{T}} .
\end{equation}
In particular, 
$$f_j \ge f_0 - |d_T^j| \ge f_0 - {\modch 2} v \sqrt{\frac{\log T}{T}} \ge \frac{f_0}{2},$$
since $\sqrt{\frac{\log T}{T}}\rightarrow 0$ as $T\rightarrow \infty$, so for $T$ large enough we  have $ {\modch 2} v \sqrt{\frac{\log T}{T}} \le \frac{f_0}{2}$. Then, for any $x \in K_T^j$, using \eqref{eq: stima f0} we have 
$$\frac{1}{f_j (x)} \le \frac{2}{f_0} \le \frac{8}{c_\eta} e^{4 \eta k}.$$
Moreover, as $b_0$ is bounded, we deduce that for all $x \in K^j_T$,
\begin{equation} \label{eq: diff b punto 2 start}
|b^i_{j}(x) - b^i_{0}(x)| \le \frac{ {\modch 16} v}{c_\eta} e^{4 \eta k} \left \| b_0^i \right \|_\infty \sqrt{\frac{\log T}{T}} + \frac{8 e^{4 \eta k}}{c_\eta}\tilde{I}^i[d_T^j] (x). 
\end{equation}
We therefore need to evaluate $\tilde{I}^i[d_T^j] (x) = \tilde{I}_1^i[d_T^j] (x) + \tilde{I}_2^i[d_T^j] (x)$ on $K_T^j$. As
\begin{equation} \label{eq: norma inf deriv dT}
\bigg\| \frac{\partial d_T^j}{\partial x_j} \bigg \|_\infty \le \frac{{\modch 2} v}{H_j}\sqrt{\frac{\log T}{T}}, 
\end{equation}
it clearly follows that
\begin{equation} \label{eq: I1 on KT}
\tilde{I}_1^i[d_T]^j (x) \le {\modch 2} k v \sqrt{\frac{\log T}{T}} \left(\frac{1}{H_1} + \frac{1}{H_2}\right). 
\end{equation}
Regarding $\tilde{I}_2^i[d_T^j] (x)$, we can act exactly as in the third point of Proposition 4 of \cite{lowerbound}. As $x \in K_T^j$, $x_i \in [x_{j,i} - H_i, x_{j,i} + H_i]$ for $i =1,2$. Therefore, using also the definition of $d_T^j$, the first integral is between $x_{j,i} - H_i$ and $x_i$. We enlarge the domain of integration to $[x_{j,i} - H_i, x_{j,i} + H_i]$ and then, appealing to \eqref{eq: norma inf dT} and \eqref{eq: norma inf deriv dT}  and the fact that the intensity of the jumps is finite, we get
\begin{equation*} \label{eq: I22 on KT}\begin{split}
|\tilde{I}_{2}^i[d_T^j] (x)| &\le \int_{x_{j,i} - H_i}^{x_{j,i} + H_i} \int_{\mathbb{R}^2}|d_T^j(\tilde{w}_i) - d_T^j(\tilde{w}_{i- 1}) + (\gamma \cdot z)_i \frac{\partial}{ \partial x_i} d_T^j (w_i)| F(z) dz dw \\
&\le 2 \, (\int_{\mathbb{R}^2} F(z) dz) \int_{x_{j,i} - H_i}^{x_{j,i} + H_i} \left \| d_T^j \right \|_\infty dw \\
&\qquad  + \int_{x_{j,i} - H_i}^{x_{j,i} + H_i} \int_{\mathbb{R}^2} \int_{\mathbb{R}^2} |(\gamma \cdot z)_i | \bigg \| \frac{\partial d_T^j}{\partial x_i} \bigg\|_\infty F(z) dz dw  \\
 &\le c H_i\sqrt{\frac{\log T}{T}} +  \frac{c H_i}{H_i} \sqrt{\frac{\log T}{T}},
\end{split}
\end{equation*}
for some $c > 0$. Using this together with \eqref{eq: diff b punto 2 start} and \eqref{eq: I1 on KT} it follows that, for any $x \in K^j_T$, 
\begin{align*}
|b_{j}(x) - b_{0}(x)| &\le 
c \sqrt{\frac{\log T}{T}} + \frac{8 e^{4 \eta k}}{c_\eta}k v \sqrt{\frac{\log T}{T}} \left(\frac{1}{H_1} + \frac{1}{H_2}\right) \\
&\qquad  + c H_i \sqrt{\frac{\log T}{T}} + c \sqrt{\frac{\log T}{T}} 
\\
& \le \frac{8 e^{4 \eta k}}{c_\eta}k v \sqrt{\frac{\log T}{T}} \left(\frac{1}{H_1} + \frac{1}{H_2}\right),
\end{align*}
where the last inequality is a consequence of the fact that, $\forall i \in \left \{ 1, 2 \right \}$, $H_i\rightarrow 0$ as $T \rightarrow \infty$ and so, for $T$ large enough, all the terms are negligible when compared to the second one.
Hence, the three points listed at the beginning of the proof hold true. We deduce that
\begin{align*}
&\int_{\mathbb{R}^2} |b_0(x) - b_j(x)|^2 f_0 (x) dx \\
&\qquad  = \int_{K_T^j} |b_0(x) - b_j(x)|^2 f_0 (x) dx + \int_{K_T^{j \, c}} |b_0(x) - b_j(x)|^2 f_0 (x) dx \\
& \qquad  \le c \, v^2 \frac{\log T}{T}H_1 H_2 + \frac{64 e^{8 \eta k}}{c_\eta^2}k^2 v^2 \frac{\log T}{T} \left(\frac{1}{H_1} + \frac{1}{H_2}\right)^2 |K_T^j|.
\end{align*}
We recall that $|K_T^j|= H_1 H_2$ and that, as $T \rightarrow \infty$, $H_i \rightarrow 0$. Thus, the first term is negligible compared to the second one. The desired result follows.

\section*{Acknowledgements}
The authors would like to thank the anonymous referees
for their helpful remarks that helped to improve the  first version of the paper.

\end{document}